\RequirePackage{fix-cm}

\documentclass[12pt]{article}

\usepackage{graphicx}
\usepackage{graphicx}%
\usepackage{multirow}%
\usepackage{amsmath,amssymb,amsfonts}%
\usepackage{amsthm}%
\usepackage{mathrsfs}%
\usepackage{xcolor}%
\usepackage{textcomp}%
\usepackage{manyfoot}%
\usepackage{booktabs}%
\usepackage{algorithm}%
\usepackage{algorithmicx}%
\usepackage{algpseudocode}%
\usepackage{listings}%
\usepackage{mathtools}
\usepackage{mathrsfs}
\usepackage{latexsym}
\usepackage{blkarray}
\usepackage{tikz}
\usepackage{comment}
\usepackage{graphicx}
\usepackage{float}
\usepackage{enumerate}
\usepackage{verbatim}
\usepackage{enumitem}

\usepackage{mathtools}
\usepackage{latexsym}
\usepackage{xcolor}
\usepackage{url}
\usepackage{hyperref}

\usepackage{multirow}

\newtheorem{assumption}{Assumption}
\newtheorem{lemma}{Lemma}
\newtheorem{theorem}{Theorem}
\newtheorem{corollary}{Corollary}
\newtheorem{example}{Example}
\newtheorem{proposition}{Proposition}

\newcommand{\nc}{\newcommand}
\newcommand{\Yildirim}{Y{\i}ld{\i}r{\i}m}

\nc{\cA}{{\cal A}}
\nc{\cB}{{\cal B}}
\nc{\cC}{{\cal C}}
\nc{\cD}{{\cal D}}
\nc{\cE}{{\cal E}}
\nc{\cG}{{\cal G}}
\nc{\cF}{{\cal F}}
\nc{\cH}{{\cal H}}
\nc{\cI}{{\cal I}}
\nc{\cK}{{\cal K}}
\nc{\cL}{{\cal L}}
\nc{\cM}{{\cal M}}
\nc{\cN}{{\cal N}}
\nc{\cO}{{\cal O}}
\nc{\cP}{{\cal P}}
\nc{\cQ}{{\cal Q}}
\nc{\cR}{{\cal R}}
\nc{\cS}{{\cal S}}
\nc{\cT}{{\cal T}}
\nc{\cV}{{\cal V}}
\nc{\tx}{{\tilde x}}
\nc{\la}{{\langle}}
\nc{\ra}{{\rangle}}
\nc{\ts}{\textsuperscript}
\nc{\R}{\mathbb{R}}
\nc{\N}{\mathbb{N}}
\nc{\Q}{\mathbb{Q}}
\nc{\Z}{\mathbb{Z}}

\DeclareMathOperator{\diag}{diag}
\DeclareMathOperator{\rank}{rank}

\DeclareMathOperator{\trace}{trace}
\DeclareMathOperator{\relint}{relint}
\DeclareMathOperator{\intK}{int}

\begin{document}

\title{Relaxations of KKT Conditions do not Strengthen Finite RLT and SDP-RLT Bounds for Nonconvex Quadratic Programs
}


\author{E.~Alper \Yildirim\thanks{School of Mathematics and Maxwell Institute of Mathematical Sciences,
The University of Edinburgh, Peter Guthrie Tait Road, Edinburgh, EH9 3FD, United Kingdom. 
    \tt{E.A.Yildirim@ed.ac.uk}}}



\date{June 11, 2025}

\maketitle
\newcommand{\keywords}[1]{\textbf{Keywords:} #1}

\newcommand{\subclass}[1]
{\textbf{Mathematics Subject Classification (2020):} #1}

\begin{abstract}
We consider linear and semidefinite programming relaxations of nonconvex quadratic programs given by the reformulation-linearization technique (RLT relaxation), and the Shor relaxation combined with the RLT relaxation (SDP-RLT relaxation). By incorporating the first-order optimality conditions, a quadratic program can be formulated as an optimization problem with complementarity constraints. We investigate the effect of incorporating optimality conditions on the strength of the RLT and SDP-RLT relaxations. Under the assumption that the original relaxations have a finite optimal value, we establish that the RLT and SDP-RLT bounds arising from the complementarity formulation agree with their counterparts. We present several classes of instances of quadratic programs with unbounded RLT and SDP-RLT relaxations to illustrate the different behavior of the relaxations of the complementarity formulation. In particular, our examples reveal that relaxations of optimality conditions may even yield misleading information on certain families of instances.

\medskip

\noindent

\medskip

\noindent
\subclass{90C20, 90C26, 90C05, 90C22, 90C46}
\end{abstract}

\section{Introduction}\label{sec1}

In this paper, we consider the problem of minimizing a (possibly nonconvex) quadratic function over a polyhedron:
\[
\tag{$\mathsf{QP}$} \label{QP} \quad \nu^* = \min\limits_{x \in \R^n} \left\{\mathsf{q}(x): x \in \mathsf{F} \right\},
\]
where $\mathsf{q}: \R^n \to \R$ and $\mathsf{F} \subseteq \R^n$ are given by
\begin{align}
\mathsf{q}(x) & = \textstyle\frac{1}{2} x^\top Q x + c^\top x, \label{def_qx} \\ 
\mathsf{F} & = \left\{x \in \R^n: G^\top x \leq g, \quad H^\top x = h \right\}. \label{def_F}
\end{align}

\eqref{QP} is a quadratic program defined by the parameters $Q \in \R^{n \times n}$, $c \in \R^n$, $G \in \R^{n \times m}$, $H \in \R^{n \times p}$, $g \in \R^m$, and $h \in \R^p$. Without loss of generality, we assume that $Q$ is a symmetric matrix. We denote the optimal value of \eqref{QP} by $\nu^* \in \R \cup \{-\infty\} \cup \{+\infty\}$, with the usual conventions for infeasible and unbounded problems.

Quadratic programs play a central role in optimization due to numerous applications (see, e.g.,~\cite{FuriniTBFGGLLMM19}). Unless $\mathsf{q}(x)$ is a convex function, \eqref{QP} is an NP-hard problem (see, e.g.,~\cite{Sahni74,Pardalos1991QuadraticPW}).

Global solution algorithms for \eqref{QP} are typically based on branch-and-bound methods, which recursively partition the feasible region into smaller subregions and generate upper and lower bounds on the objective function value in each subregion (see, e.g.,~\cite{LS2013}). Lower bounds are usually given by solving convex relaxations of \eqref{QP}. Strong relaxations provide tighter lower bounds, which can potentially improve the solution time by closing the gap earlier.  

In this paper, we consider two convex relaxations of \eqref{QP}. The first one, referred to as the RLT relaxation, is a linear programming relaxation given by the reformulation-linearization technique (see, e.g., ~\cite{Sherali1999}). The RLT relaxation is obtained by generating implied quadratic constraints for $\mathsf{F}$ and then replacing each quadratic term with a new variable. The second relaxation, referred to as the SDP-RLT relaxation, is a semidefinite programming relaxation obtained by combining the RLT relaxation with the Shor relaxation~\cite{shor1987approach}, further tightening the relationship between the original variables and the new variables representing quadratic terms. 

By incorporating the first-order optimality conditions, \eqref{QP} can be formulated as an optimization problem with complementarity constraints (see, e.g.,~\cite{GT73}). Unless \eqref{QP} is unbounded, the two formulations are equivalent. On the other hand, it is well-known that convex relaxations arising from equivalent formulations may not, in general, be equivalent (see, e.g.,~\cite[Example 3]{QiuY23a}).  

Motivated by this observation, we investigate the effect of incorporating the first-order optimality conditions on the strength of RLT and SDP-RLT relaxations of \eqref{QP}. Our contributions are as follows:

\begin{enumerate}

    \item If the RLT relaxation of \eqref{QP} has a finite optimal value, we show that incorporating optimality conditions does not improve the strength of the RLT relaxation (see Theorem~\ref{rltcserlt}).
    
    \item Unless $p = 0$, we show that the Slater condition fails for the SDP-RLT relaxation (see Lemma~\ref{Slater}). 

    \item Using facial reduction techniques, we present an equivalent formulation of the SDP-RLT relaxation as a conic optimization problem that satisfies the Slater condition (see Proposition~\ref{srvssrr}). 

    \item By relying on the equivalence between the two formulations, we establish necessary and sufficient optimality conditions for the original SDP-RLT relaxation (see Proposition~\ref{opt_cond_srr}).

    \item We establish that incorporating optimality conditions does not improve the strength of the SDP-RLT relaxation if its optimal value is attained (see Theorem~\ref{srltcsesrlt}).

    \item We present several classes of instances of \eqref{QP} with unbounded RLT and SDP-RLT relaxations to illustrate the different behavior of the relaxations arising from the complementarity formulation. In particular, our examples reveal that the information from the relaxations of optimality conditions should be treated with some caveat.
\end{enumerate}

We briefly motivate our choice of $\mathsf{F}$ in the general form given by \eqref{def_F}. RLT and SDP-RLT relaxations are, in general, highly dependent on the particular representation of $\mathsf{F}$. For instance, if one replaces each equality constraint with two inequality constraints, then the resulting RLT relaxation is, in general, weaker than the original one (see, e.g., the discussion at the end of Section 6 in~\cite{QiuY23a}). Furthermore, converting $\mathsf{F}$ into a more convenient form may increase the dimension of the problem, which may adversely affect the computational cost of the relaxations. Therefore, throughout this manuscript, we will assume the general form given by \eqref{def_F}.

This paper is organized as follows. We review the literature in Section~\ref{lit_rev} and define our notation in Section~\ref{notation}. We present the first-order optimality conditions and incorporate these conditions to obtain a reformulation of \eqref{QP} as an optimization problem with complementarity constraints in Section~\ref{Sec2}. Section~\ref{Sec3} is devoted to the comparison of the RLT relaxations arising from the original formulation and the complementarity formulation. We compare the corresponding SDP-RLT relaxations in Section~\ref{Sec4}. Section~\ref{Sec5} concludes the paper.

\subsection{Literature Review} \label{lit_rev}

In addition to playing a central role in guiding algorithmic methods and their termination criteria, optimality conditions can be taken into consideration prior to solving an optimization problem. For instance, by directly incorporating the first-order optimality conditions, \eqref{QP} can be reformulated as a linear programming problem with complementarity constraints~\cite{GT73}. Such reformulations can then be utilized to develop solution methods based on mixed integer linear programming approaches (see, e.g.,~\cite{HuMP12,XiaVZ20,GondzioY21}).

Optimality conditions can also be employed to develop branching strategies in a branch-and-bound framework. For box-constrained quadratic programs, Hansen et al.~\cite{HPBR93} propose a finite branch-and-bound algorithm by employing first-order optimality conditions to subdivide the feasible region (see also Audet et al.~\cite{AudetHJS99} for a similar method for disjoint bilinear programming), and Vandenbussche and Nemhauser~\cite{VandenbusscheN05} utilize the optimality conditions from~\cite{HPBR93} to generate valid inequalities for a linear programming relaxation and develop a finite branch-and-bound algorithm using a similar branching rule~\cite{VandenbusscheN05a}. By using semidefinite programming relaxations, Burer and Vandenbussche propose a similar finite branch-and-bound method for box-constrained quadratic programs~\cite{BurerV09} and general quadratic programs with bounded feasible regions~\cite{BurerV08} (see also \cite{ChenB12} for a copositive relaxation-based approach). In addition to the first-order optimality conditions, Burer and Chen~\cite{BurerC11} propose semidefinite programming relaxations that incorporate additional information from second-order optimality conditions in a branch-and-bound framework for box-constrained quadratic programs.

In this paper, our focus is on the effect of incorporating optimality conditions on the strength of convex relaxations. The relaxations considered in this paper are similar to those in~\cite{VandenbusscheN05,BurerV08,BurerV09,BurerC11,BaoST11,ChenB12} that utilize optimality conditions in various ways. In particular, we compare the RLT and SDP-RLT relaxations of \eqref{QP} with their counterparts arising from the complementarity formulation obtained by adding the first-order optimality conditions. In fact, the resulting SDP-RLT relaxation is equivalent to the doubly nonnegative relaxation considered in the branch-and-bound method of~\cite{ChenB12} (see, e.g.,~\cite[Proposition 9]{BaoST11}). 

In particular, our results are closely related to and build on~\cite{BurerC11}, where they establish the equivalence of the basic Shor relaxation and various other semidefinite relaxations arising from the incorporation of first- and second-order optimality conditions for box-constrained quadratic programs. In this paper, we consider general quadratic programs, whereas the results of~\cite{BurerC11} are applicable to box-constrained quadratic programs. In contrast with~\cite{BurerC11}, which considers variations of the basic Shor relaxation, we consider tighter semidefinite programming relaxations given by combining the Shor relaxation with the RLT relaxation of the complementarity formulation. On the other hand, while the relaxations of~\cite{BurerC11} also incorporate additional information from the second-order optimality conditions, our relaxations utilize only the first-order conditions.

\subsection{Notation} \label{notation}

We use $\R^n$, $\R^n_+$, $\R^n_{++}$, $\R^{m \times n}$, $\cS^n$, $\cS^n_+$, and $\cS^n_{++}$, to denote the $n$-dimensional Euclidean space, the nonnegative orthant, the positive orthant, the set of $m \times n$ real matrices, the space of $n \times n$ real symmetric matrices, the cone of $n \times n$ positive semidefinite matrices, and the cone of $n \times n$ positive definite matrices, respectively. For $U \in \cS^n$, we also use $U \succeq \mathbf{0}$ and $U \succ \mathbf{0}$ to denote that $U$ is positive semidefinite and positive definite, respectively.  We use 0 to denote the real number zero, whereas the vector of all zeroes, as well as the matrix of all zeroes are denoted by $\mathbf{0}$, whose dimension should always be clear from the context. We denote matrices, vectors, and scalars by uppercase Roman letters, lowercase Roman letters, and Greek letters, respectively, with the exceptions of $i$ and $j$ used for indexing purposes and $m, n$, and $p$ reserved for dimensions. We denote subsets of $\cS^n$ by calligraphic letters. For all other objects such as sets, functions, and problem labels, we use the sans serif font type. All inequalities on vectors or matrices are understood to hold componentwise. We reserve $I$ for the identity matrix, whose dimension should be implied by the context. The transpose of $U \in \R^{m \times n}$ is denoted by $U^\top$. We denote by $\diag(U) \in \R^n$ the vector obtained from the diagonal entries of $U \in \R^{n \times n}$. 
The rank of a matrix $A \in \R^{m \times n}$ is denoted by $\rank(A)$. 
For $u \in \R^n$ and $U \in \R^{m \times n}$, we use $u_j$ and $U_{ij}$ to denote the $i$th component of $u$ and the $(i,j)$-component of $U$, respectively. For $U \in \R^{m \times n}$ and $V \in \R^{m \times n}$, the trace inner product is denoted by 
\[
\langle U, V \rangle = \trace(U^\top V) = \sum\limits_{i=1}^m \sum\limits_{j = 1}^n U_{ij} V_{ij}.
\]
For a given closed convex cone $\cK \subseteq \cS^n$, the dual cone, denoted by $\cK^*$, is given by
\[
\cK^* = \left\{Y \in \cS^n: \langle X, Y \rangle \geq 0, \quad \text{for all}~X \in \cK \right\}.
\]
We use $\relint(\cK)$ and $\intK(K)$ to denote the relative interior and interior of $\cK$, respectively.

Throughout this manuscript, we adopt the following conventions for the linear algebra of empty matrices and empty vectors~\cite{deB90}. If $U \in \R^{m \times 0}$ is an $m \times 0$ empty matrix, $v \in \R^{0}$ is a $0 \times 1$ empty vector, and $V \in \R^{0 \times n}$ empty matrix, then $U \, v = \mathbf{0} \in \R^{m}$ and $U \, V = \mathbf{0} \in \R^{m \times n}$, whereas we use the usual linear algebra rules for the product of an empty matrix and a nonempty matrix or vector of conforming dimensions, i.e., the product of a $0 \times m$ empty matrix with and $m \times n$ matrix is a $0 \times n$ empty matrix. Finally, we use
\begin{equation} \label{emptyinv}
U \in \R^{0 \times 0} \Longrightarrow U^{-1} = U.
\end{equation}

\section{Optimality Conditions and a Formulation with Complementarity Constraints} \label{Sec2}

In this section, we review the first-order optimality conditions for \eqref{QP}. By incorporating the optimality conditions, we present a formulation of \eqref{QP} as an optimization problem with complementarity constraints.

Throughout this paper, we make the following assumption.

\begin{assumption} \label{assump1}
$\mathsf{F} \neq \emptyset$ and $\rank(H) = p < n$. Furthermore, there exists $x^\circ \in \mathsf{F}$ such that $G^\top x^\circ < g$ and $H^\top x^\circ = h$.    
\end{assumption}

We note that the feasibility of \eqref{QP} can be checked via linear programming. Similarly, the other assumptions can easily be satisfied by a simple preprocessing stage. Finally, if $\rank(H) = p = n$, then $\mathsf{F} = \{x^\circ\}$, in which case the feasible region of each of the RLT relaxation and the SDP-RLT relaxation consists of the unique point $(x^\circ, x^\circ \, (x^\circ)^\top)$~\cite[Corollary 28]{QiuY23a}. Therefore, the lower bound from each of the two relaxations already agrees with the optimal value $\nu^*$.  

\subsection{Optimality Conditions}

Since the feasible region of \eqref{QP} is a nonempty polyhedron, constraint qualification holds at each feasible solution. Therefore, if $x \in \R^n$ is a local minimizer of \eqref{QP}, then there exists $(y,z) \in \R^m \times \R^p$ such that the following first-order (KKT) optimality conditions are satisfied: 
\begin{align}
Q \, x + c + G \, y + H \, z & = \mathbf{0} \label{kkt1} \\
G^\top x & \leq g \label{kkt2} \\
H^\top x & = h \label{kkt3} \\
y^\top (g - G^\top x) & = 0 \label{kkt4} \\
y & \geq \mathbf{0}. \label{kkt5}
\end{align}
We say that $x \in \R^n$ is a KKT point (or critical point) of \eqref{QP} if there exists $(y,z) \in \R^m \times \R^p$ satisfying \eqref{kkt1}--\eqref{kkt5}. The set of all KKT points is denoted by
\begin{equation} \label{kktp}
\mathsf{F}^{\mathsf{c}} = \{x \in \R^n: \exists~(y,z) \in \R^m \times \R^p \textrm{ s.t. }\eqref{kkt1}-\eqref{kkt5} \textrm{ hold.}\}  
\end{equation}

If $x \in \R^n$ is a KKT point of \eqref{QP}, then it follows from \eqref{kkt1}, \eqref{kkt3}, and \eqref{kkt4} that (see also \cite{GT73})
\begin{equation} \label{kktobj}
x^\top Q \, x + c^\top x + g^\top y + h^\top z = 0.
\end{equation}

\subsection{Complementarity Formulation}

In this section, we present a formulation of \eqref{QP} by incorporating the first-order optimality conditions. 

By adding the optimality conditions \eqref{kkt1}, \eqref{kkt4}, and \eqref{kkt5} as well as the implied relation \eqref{kktobj} to \eqref{QP}, we obtain the following reformulation with complementarity constraints: 
\[
\begin{array}{llrcl}
\tag{$\mathsf{QP+}$} \label{eqp} &  \nu^*_{\mathsf{+}} = \min\limits_{x \in \R^n, y \in \R^m, z \in \R^p} & \textstyle\frac{1}{2} x^\top Q \, x + c^\top x & & \\
 & \textrm{s.t.} & & & \\
 & & G^\top x & \leq & g\\
 & & H^\top x & = & h \\
 & & Q \, x + c + G \, y + H \, z & = & \mathbf{0} \\
 & & y^\top (g - G^\top x) & = & 0 \\
 & & y & \geq & \mathbf{0} \\
  & & x^\top Q \, x + c^\top x + g^\top y + h^\top z & = & 0.
\end{array}
\]

By using \eqref{kktobj}, it is easy to verify that 
\[
\textstyle\frac{1}{2} x^\top Q \, x + c^\top x = \textstyle\frac{1}{2} \left( c^\top x - g^\top y - h^\top z \right),
\]
which implies that \eqref{eqp} can be formulated as a linear programming problem with complementarity constraints (see also~\cite{GT73}). However, we will continue to use the original objective function in \eqref{eqp}, which will facilitate a direct comparison of the relaxations arising from \eqref{QP} and \eqref{eqp}. 

We review several useful results about the complementarity formulation \eqref{eqp}.

\begin{lemma} \label{eqp_props}
\begin{enumerate}
    \item[(i)] The feasible region of \eqref{eqp} is nonempty if and only if $\mathsf{F}^{\mathsf{c}} \neq \emptyset$, where $\mathsf{F}^{\mathsf{c}}$ is defined as in \eqref{kktp}.
    \item[(ii)] If \eqref{eqp} is not infeasible, then $\nu^*_{\mathsf{+}} > -\infty$ and $\nu^*_{\mathsf{+}}$ is equal to the minimum objective function value of any KKT point of \eqref{QP}.
    \item[(iii)] If $\nu^* > -\infty$, then $\nu^*_{\mathsf{+}} = \nu^*$.
\end{enumerate} 
\end{lemma}
\begin{proof}
\begin{enumerate}
    \item[(i)] If $(x,y,z)$ is \eqref{eqp}-feasible, then $x \in \mathsf{F}^{\mathsf{c}}$. The converse implication follows from \eqref{kktp} and the observation that \eqref{kktobj} is implied by the KKT conditions. 
    \item[(ii)] By part (i), $\mathsf{F}^{\mathsf{c}} \neq \emptyset$. The assertion follows from~\cite[Lemma 3.1]{LuoT92}. 
    \item[(iii)] If $\nu^* > -\infty$, then the set of optimal solutions of \eqref{QP} is nonempty~\cite{FW1956}, which implies that $\mathsf{F}^{\mathsf{c}} \neq \emptyset$. The assertion follows from (i) and (ii).
\end{enumerate} 

\end{proof}

\section{RLT Relaxations} \label{Sec3}

In this section, we consider RLT (reformulation-linearization technique) relaxations. First, we present the RLT relaxation of \eqref{QP} and briefly review its optimality conditions. We then consider the RLT relaxation of the complementarity formulation \eqref{eqp}. We finally compare the strengths of the two RLT relaxations under different scenarios.

\subsection{RLT Relaxation and Optimality Conditions}

The RLT relaxation of \eqref{QP} is a linear programming relaxation given by a two-stage approach (see, e.g.,~\cite{Sherali1999}). In the first stage, one generates implied quadratic constraints on $\mathsf{F}$ by treating linear equality and inequality constraints in different ways. The quadratic inequality constraints are obtained by the products of all pairs of linear inequality constraints. On the other hand, the quadratic equality constraints are generated by multiplying each linear equality constraint by each variable. We remark that all other quadratic constraints obtained by the product of two linear equality constraints, or the product of a linear equality constraint and a linear inequality constraint, are already implied by this procedure (see, e.g.,~\cite[Remark 1]{sherali1995reformulation}, \cite{QiuY23a}). In the second stage, all quadratic constraints are linearized by replacing each quadratic term $x_i x_j$ with a new variable $X_{ij},~i = 1,\ldots,n; j = 1,\ldots,n$. 

The RLT relaxation of \eqref{QP} is therefore given by
\[
\begin{array}{llrcl}
\tag{$\mathsf{R}$} \label{rlt} & \nu^*_{\mathsf{R}} = \min\limits_{x \in \R^n, X \in \cS^n} & \textstyle\frac{1}{2}\langle Q, X \rangle + c^\top x & & \\
 & \textrm{s.t.} & & & \\
 & & G^\top x & \leq & g\\[0.2em]
 & & H^\top x & = & h \\[0.2em]
 & & H^\top X & = & h \, x^\top \\[0.2em]
 & & G^\top X G - G^\top x \, g^\top - g \, x^\top G + g \, g^\top & \geq & \mathbf{0}. 
\end{array}
\]

Note that \eqref{rlt} is a relaxation of \eqref{QP} since, for each $x \in \mathsf{F}$, $(x, x \, x^\top)$ is \eqref{rlt}-feasible with the same objective function value, where $\mathsf{F}$ is given by \eqref{def_F}. Therefore, 
\begin{equation} \label{rlt-lb}
    \nu^*_{\mathsf{R}} \leq \nu^*.
\end{equation}

Let $(u,w,R,S) \in \R^m \times \R^p \times \R^{p \times n} \times \cS^m$ denote the dual variables corresponding to the four sets of constraints in \eqref{rlt}, respectively. After scaling $S$ by a factor of $\textstyle\frac{1}{2}$, the dual of \eqref{rlt} is given by (see, e.g.,~\cite{QiuY23a})
\[
\begin{array}{llrcl}
\tag{$\mathsf{RD}$} \label{rltd} & \max\limits_{u \in \R^m, w \in \R^p, R \in \R^{p \times n}, S \in \cS^m} & -u^\top g + w^\top h - \textstyle\frac{1}{2}g^\top S \, g & & \\
 & \textrm{s.t.} & & & \\
 & & -G \, u + H \, w - R^\top h - G \, S \, g & = & c\\[0.2em]
 & & R^\top H^\top + H \, R + G \, S \, G^\top & = & Q \\[0.2em]
 & & S & \geq & \mathbf{0} \\[0.2em]
 & & u & \geq & \mathbf{0}.
\end{array}
\]

Using standard linear programming duality, we arrive at the following optimality conditions~(see, e.g.,~\cite[Lemma 18]{QiuY23a}).

\begin{lemma} \label{opt_cond}
Let $(\hat x, \hat X) \in \R^n \times \cS^n$ be \eqref{rlt}-feasible. Then, $(\hat x, \hat X)$
 is an optimal solution of \eqref{rlt} if and only if there exists 
$(\hat u, \hat w, \hat R, \hat S) \in \R^m \times \R^p \times \R^{p \times n} \times \cS^m$ such that
\begin{align} 
-G \, \hat u + H \, \hat w - \hat R^\top h - G \, \hat S \, g & = c, \label{opt_c} \\
\hat R^\top H^\top + H \, \hat R + G \, \hat S \, G^\top & = Q, \label{opt_Q} \\
\hat u & \geq \mathbf{0}, \label{opt_u} \\
\hat S & \geq \mathbf{0}, \label{opt_S} \\
\hat u^\top \left( g - G^\top \hat x \right) & = 0, \label{opt_cs1} \\
\left\langle \hat S, G^\top \hat X G - G^\top \hat x \, g^\top - g \, \hat x^\top G + g \, g^\top \right\rangle & = 0. \label{opt_cs2}
\end{align}
\end{lemma}

\subsection{RLT Relaxation of the Complementarity Formulation}

The RLT relaxation of \eqref{eqp} is given by
\[
\tag{$\mathsf{R+}$} \label{erlt}
\begin{array}{llrcl}
& \nu^*_{\mathsf{R+}} = \min\limits_{\substack{x \in \R^n, y \in \R^m, z \in \R^p, \\ X \in \cS^n, Y \in \cS^m, Z \in \cS^p, \\ M_{xy} \in \R^{n \times m}, M_{xz} \in \R^{n \times p}, M_{yz} \in \R^{m \times p}}} & \textstyle\frac{1}{2}\langle Q, X \rangle + c^\top x & & \\
 & \textrm{s.t.} & & & \\
 & & G^\top x & \leq & g\\[0.2em]
 & & H^\top x & = & h \\[0.2em]
  & & H^\top X & = & h \, x^\top \\[0.2em]
   & & G^\top X G - G^\top x \, g^\top - g \, x^\top G + g \, g^\top & \geq & \mathbf{0} \\[0.2em]
 & & Q \, x + c + G \, y + H \, z & = & \mathbf{0}\\[0.2em]
  & & H^\top M_{xy} & = & h \, y^\top \\[0.2em]
  & & H^\top M_{xz} & = & h \, z^\top \\[0.2em]
 & & Q \, X + c \, x^\top + G \, M_{xy}^\top + H \, M_{xz}^\top & = & \mathbf{0} \\[0.2em]
 & & Q \, M_{xy} + c \, y^\top + G \, Y + H \, M_{yz}^\top & = & \mathbf{0} \\[0.2em]
 & & Q \, M_{xz} + c \, z^\top + G \, M_{yz} + H \, Z  & = & \mathbf{0}\\[0.2em]
 & & \diag \left( y \, g^\top - M_{xy}^\top \, G \right) & = & \mathbf{0}\\[0.25em]
 & & y \, g^\top - M_{xy}^\top \, G & \geq & \mathbf{0} \\[0.2em]
 & & y & \geq & \mathbf{0} \\[0.2em]
 & & Y & \geq & \mathbf{0} \\[0.2em]
 & & \langle Q, X \rangle + c^\top x + g^\top y + h^\top z & = & 0.
\end{array}
\]

In \eqref{erlt}, $X, M_{xy}, M_{xz}, Y, M_{yz}$, and $Z$ are the variables corresponding to the linearizations of the quadratic terms $x \, x^\top$, $x \, y^\top$, $x \, z^\top$, $y \, y^\top$, $y \, z^\top$, and $z \, z^\top$, respectively.  The first four constraints are given by the RLT procedure applied to the constraints from the original formulation. The sixth and seventh constraints are obtained by multiplying the second constraint of \eqref{eqp} by $y^\top$ and $z^\top$, respectively, and the eighth, ninth, and tenth constraints are obtained by multiplying the third constraint of \eqref{eqp} by $x^\top, y^\top$, and $z^\top$, respectively. The eleventh constraint is given by the linearization of the fourth constraint of \eqref{eqp}. We obtain the twelfth constraint by linearizing the product of $G^\top x \leq g$ and $y \geq \mathbf{0}$, and the fourteenth constraint by linearizing the product of $y \geq \mathbf{0}$ with itself. Finally, the last constraint is the linearization of the last quadratic equality in \eqref{eqp}.

\subsection{Comparison of Two RLT Relaxations: Finite RLT Bound}

In this section, under the assumption that $\nu^*_{\mathsf{R}}$ is finite, we compare the strength of the RLT relaxations \eqref{rlt} and \eqref{erlt} obtained from the original formulation \eqref{QP} and the complementarity formulation \eqref{eqp}, respectively.

\begin{theorem} \label{rltcserlt}
Suppose that $\nu^*_{\mathsf{R}} > - \infty$. Then, the RLT relaxations  \eqref{rlt} and \eqref{erlt} are equivalent, i.e., $\nu^*_{\mathsf{R}} = \nu^*_{\mathsf{R+}}$. 
\end{theorem}
\begin{proof}
Clearly, $\nu^*_{\mathsf{R}} \leq \nu^*_{\mathsf{R+}}$ since \eqref{erlt} has additional constraints. 

To establish the reverse inequality, suppose that $\nu^*_{\mathsf{R}} > - \infty$. Let $(\hat x, \hat X) \in \R^n \times \cS^n$ denote an optimal solution of \eqref{rlt}. We will construct
\[
(\hat y, \hat z, \hat Y, \hat Z, \hat M_{xy}, \hat M_{xz}, M_{yz}) \in \R^m  \times \R^p  \times \cS^m  \times \cS^p  \times \R^{n \times m} \times \R^{n \times p} \times \R^{m \times p}
\]
such that $(\hat x, \hat y, \hat z, \hat X, \hat Y, \hat Z, \hat M_{xy}, \hat M_{xz}, \hat M_{yz})$ is \eqref{erlt}-feasible. 

By the optimality conditions for \eqref{rlt}, there exists $(\hat u, \hat w, \hat R, \hat S) \in \R^m \times \R^p \times \R^{p \times n} \times \cS^m$ such that \eqref{opt_c}--\eqref{opt_cs2} hold. 

By \eqref{opt_c} and \eqref{opt_Q}, we obtain
\begin{equation} \label{erltthird}
Q \, \hat x + c + G \left( \hat S \, (g - G^\top \hat x) + \hat u \right) + H \left( - \hat R \, \hat x - \hat w \right) = \mathbf{0}.
\end{equation}

In view of the fifth constraint of \eqref{erlt}, we define
\begin{align} 
\hat y & = \hat S \, (g - G^\top \hat x) + \hat u, \label{opty} \\
\hat z & = - \hat R \, \hat x - \hat w. \label{optz}
\end{align}

We define each of the remaining lifted variables by using the corresponding rank-one matrix, except that we replace $\hat x \, \hat x^\top$ by $\hat X$:
\begin{align} 
\hat Y & = \hat S \left( G^\top \hat X G - G^\top \hat x \, g^\top - g \, \hat x^\top G + g \, g^\top \right) \hat S \nonumber \\ 
 & \quad + \hat S \left( g - G^\top \hat x \right) \hat u^\top + \hat u \left(g - G^\top \hat x \right)^\top \hat S + \hat u \, \hat u^\top, \label{optY} \\
\hat Z & = \hat R \, \hat X \hat R^\top + \hat R \, \hat x \, \hat w^\top + \hat w \, \hat x^\top \hat R^\top + \hat w \, \hat w^\top, \label{optZ} \\
\hat M_{xy} & = \hat x \, g^\top \hat S - \hat X \, G \, \hat S + \hat x \, \hat u^\top, \label{optU} \\
\hat M_{xz} & = - \hat X \, \hat R^\top - \hat x \, \hat w^\top, \label{optV} \\
\hat M_{yz} & = - \hat S \, g \, \hat x^\top \hat R^\top - \hat S \, g \, \hat w^\top + \hat S \, G^\top \hat X \hat R^\top + \hat S \, G^\top \hat x \, \hat w^\top - \hat u \, \hat x^\top \hat R^\top - \hat u \, \hat w^\top. \label{optW} 
\end{align}

We claim that $(\hat x, \hat y, \hat z, \hat X, \hat Y, \hat Z, \hat M_{xy}, \hat M_{xz}, \hat M_{yz})$ is \eqref{erlt}-feasible. The first four constraints of \eqref{erlt} are satisfied since $(\hat x, \hat X)$ is \eqref{rlt}-feasible, and the fifth constraint holds by \eqref{erltthird}, \eqref{opty}, and \eqref{optz}.

Let us first consider the sixth and seventh constraints:
\begin{align*}
H^\top \hat M_{xy} & \stackrel{\eqref{optU}}{=} h \, g^\top \hat S - h \, \hat x^\top \, G \, \hat S + h \, \hat u^\top \stackrel{\eqref{opty}}{=} h \, \hat y^\top, \\
H^\top \hat M_{xz} & \stackrel{\eqref{optV}}{=} - h \, \hat x^\top \hat R^\top - h \, \hat w^\top \stackrel{\eqref{optz}}{=} h \, \hat z^\top, 
\end{align*}
where we used $H^\top \hat x = h$ and $H^\top \hat X = h \, \hat x^\top$.

Considering the eighth constraint, we have 
\begin{align*}
Q \, \hat X & \stackrel{\eqref{opt_Q}}{=} \hat R^\top h \, \hat x^\top + H \hat R \, \hat X + G \, \hat S \, G^\top \hat X, \\
c \, \hat x^\top & \stackrel{\eqref{opt_c}}{=} -G \, \hat u \, \hat x^\top + H \, \hat w \, \hat x^\top - \hat R^\top h \, \hat x^\top - G \, \hat S \, g \, \hat x^\top, \\ 
G \, \hat M_{xy}^\top & \stackrel{\eqref{optU}}{=} G \, \hat S \, g \, \hat x^\top - G \, \hat S \, G^\top \hat X + G \, \hat u \, \hat x^\top , \\ 
H \hat M_{xz}^\top & \stackrel{\eqref{optV}}{=} - H \hat R \, \hat X - H \, \hat w \, \hat x^\top ,   
\end{align*}
where we used $H^\top \hat X = h \, \hat x^\top$ in the first line. 
Therefore,  
$Q \, \hat X + c \, \hat x^\top + G \, \hat M_{xy}^\top + H \, \hat M_{xz}^\top = \mathbf{0}$.

Next, we focus on the ninth constraint:
\begin{align*}
Q \, \hat M_{xy} \stackrel{\eqref{opt_Q}, \eqref{optU}}{=} &
\hat R^\top h \, g^\top \hat S - \hat R^\top h \, \hat x^\top G \, \hat S + \hat R^\top h \, \hat u^\top + H \hat R \, \hat x \, g^\top \hat S - H \hat R \, \hat X G \, \hat S  \\ 
 & \quad  + \, H \hat R \, \hat x \, \hat u^\top + G \, \hat S \, G^\top \hat x \, g^\top \hat S - G \, \hat S \, G^\top \hat X G \, \hat S + G \, \hat S \, G^\top \hat x \, \hat u^\top, \\
c \, \hat y^\top  \stackrel{\eqref{opt_c}, \eqref{opty}}{=} & - G \, \hat u \, g^\top \hat S + G \, \hat u \, \hat x^\top G \, \hat S - G \, \hat u \, \hat u^\top    + H \, \hat w \, g^\top \hat S - H \, \hat w \, \hat x^\top G \, \hat S \\ 
 & \quad + H \, \hat w \, \hat u^\top 
 - \, \hat R^\top h \, g^\top \hat S + \hat R^\top h \, \hat x^\top G \, \hat S - \hat R^\top h \, \hat u^\top - G \, \hat S \, g \, g^\top \hat S \\
 & \quad + G \, \hat S \, g  \, \hat x^\top G \, \hat S - G \, \hat S \, g \, \hat u^\top , \\
G \, \hat Y \stackrel{\eqref{optY}}{=} & G \, \hat S \, G^\top \hat X \, G \, \hat S - G \, \hat S \, G^\top \hat x \, g^\top \hat S - G \, \hat S \, g \, \hat x^\top G \, \hat S + G \, \hat S \, g \, g^\top \hat S   \\ 
 & \quad + \, G \, \hat S \, g \, \hat u^\top - G \, \hat S \, G^\top \hat x \, \hat u^\top + G \, \hat u \, g^\top \hat S -  G \, \hat u \, \hat x^\top G \, \hat S  + G \, \hat u \, \hat u^\top, \\
H \hat M_{yz}^\top  \stackrel{\eqref{optW}}{=} & - H \hat R \, \hat x \, g^\top \hat S - H \, \hat w \, g^\top \hat S 
+ H \hat R \, \hat X \, G \, \hat S \\
 & \quad + H \, \hat w \, \hat x^\top G \, \hat S - H \hat R \, \hat x \, \hat u^\top - H \, \hat w \, \hat u^\top,
\end{align*}
where we used $H^\top \hat x = h$ and $H^\top \hat X = h \, \hat x^\top$ in the first equality. 
It follows that 
$Q \, \hat M_{xy} + c \, \hat y^\top + G \, \hat Y + H \hat M_{yz}^\top = \mathbf{0}$.

Considering the tenth constraint, we obtain
\begin{align*}
Q \, \hat M_{xz}  \stackrel{\eqref{opt_Q}, \eqref{optV}}{=} & - \hat R^\top h \, \hat x^\top \hat R^\top - \hat R^\top h \, \hat w^\top - H \hat R \, \hat X \hat R^\top - H \hat R \, \hat x \, \hat w^\top \\ 
 & \quad - \, G \, \hat S \, G^\top \hat X \hat R^\top - G \, \hat S \, G^\top \hat x \, \hat w^\top, \\
c \, \hat z^\top \stackrel{\eqref{opt_c},\eqref{optz}}{=} & G \, \hat u \, \hat x^\top \hat R^\top + G \, \hat u \, \hat w^\top - H \, \hat w \, \hat x^\top \hat R^\top - H \, \hat w \, \hat w^\top \\
 &  \quad + \, \hat R^\top h \, \hat x^\top \hat R^\top + \hat R^\top h \, \hat w^\top + G \, \hat S \, g \, \hat x^\top \hat R^\top + G \, \hat S \, g \, \hat w^\top, \\
G \, \hat M_{yz} \stackrel{\eqref{optW}}{=} & - G \, \hat S \, g \, \hat x^\top \hat R^\top - G \, \hat S \, g \, \hat w^\top 
+ G \, \hat S \, G^\top \hat X \hat R^\top + G \, \hat S \, G^\top \hat x \, \hat w^\top \\ 
 & \quad - \, G \, \hat u \, \hat x^\top \hat R^\top  - G \, \hat u \, \hat w^\top , \\ 
H \hat Z \stackrel{\eqref{optZ}}{=} & H \hat R \, \hat X \hat R^\top + H \hat R \, \hat x \, \hat w^\top + H \, \hat w \, \hat x^\top \hat R^\top + H \, \hat w \, \hat w^\top,    
\end{align*}
where we used $H^\top \hat x = h$ and $H^\top \hat X = h \, \hat x^\top$ in the first equality. We conclude that 
$Q \, \hat M_{xz} + c \, \hat z^\top + G \, \hat M_{yz} + H \hat Z = \mathbf{0}$.

We next focus on the eleventh and twelfth constraints:
\begin{eqnarray*}
\hat y \, g^\top - \hat M_{xy}^\top \, G & \stackrel{\eqref{opty},\eqref{optU}}{=} & \hat S \, g \, g^\top - \hat S \, G^\top \hat x \, g^\top + \hat u \, g^\top - \hat S \, g \, \hat x^\top G + \hat S \, G^\top \hat X G - \hat u \, \hat x^\top G  \\
  & = & \hat u \left( g - G^\top \hat x \right)^\top + \hat S \left( G^\top \hat X G - G^\top \hat x \, g^\top - g \, \hat x^\top G + g \, g^\top \right).
\end{eqnarray*}
By \eqref{opt_u}, the first constraint, \eqref{opt_S}, and the fourth constraint, we conclude that $\hat y \, g^\top - \hat M_{xy}^\top \, G \geq \mathbf{0}$. By \eqref{opt_S}, \eqref{opt_cs1},  \eqref{opt_cs2}, and the fourth constraint, we obtain $\diag \left( \hat y \, g^\top - \hat M_{xy}^\top \, G \right) = \mathbf{0}$. 

Next, it follows from \eqref{opty}, \eqref{optY}, \eqref{opt_u}, the first constraint, \eqref{opt_S}, and the fourth constraint that $\hat y \geq \mathbf{0}$ and $\hat Y \geq \mathbf{0}$. Finally, considering the last constraint, we arrive at 
\begin{align*}
    \langle Q, \hat X \rangle & \stackrel{\eqref{opt_Q}}{=} 2 \, h^\top \hat R \, \hat x + \left\langle \hat S, G^\top \hat X G \right\rangle, \\
    c^\top \hat x & \stackrel{\eqref{opt_c}}{=} - \hat u^\top G^\top \hat x + \hat w^\top h - h^\top \hat R \, \hat x - g^\top \hat S \, G^\top \hat x, \\
    g^\top \hat y & \stackrel{\eqref{opty}}{=} g^\top \hat S \, g - g^\top \hat S \, G^\top \hat x + g^\top \hat u, \\
    h^\top \hat z & \stackrel{\eqref{optz}}{=} -h^\top \hat R \, \hat x - h^\top \hat w,
\end{align*}
where we used $H^\top \hat X = h \, \hat x^\top$ in the first equality. 
It follows that
\begin{eqnarray*}
\langle Q, \hat X \rangle + c^\top \hat x + g^\top \hat y + h^\top \hat z & = & \hat u^\top \left( g - G^\top \hat x \right) \\ 
 & & \quad + 
\left\langle \hat S, G^\top \hat X G - G^\top \hat x \, g^\top - g \, \hat x^\top G + g \, g^\top \right\rangle \\
 & = & 0,   
\end{eqnarray*}
where we used \eqref{opt_cs1} and \eqref{opt_cs2} to derive the second equality. 

Therefore, $(\hat x, \hat y, \hat z, \hat X, \hat Y, \hat Z, \hat M_{xy}, \hat M_{xz}, \hat M_{yz})$ is \eqref{erlt}-feasible with the same objective function value, which implies that $\nu^*_{\mathsf{R+}} \leq \nu^*_{\mathsf{R}}$. We conclude that $\nu^*_{\mathsf{R+}} =  \nu^*_{\mathsf{R}}$.

\end{proof}

Theorem~\ref{rltcserlt} establishes the equivalence of the two RLT relaxations \eqref{rlt} and \eqref{erlt} under the assumption that $\nu^*_{\mathsf{R}}$ is finite. The following corollary presents a sufficient condition for a finite RLT lower bound $\nu^*_{\mathsf{R}}$.

\begin{corollary} \label{polytope_imp}
If $\mathsf{F}$ is nonempty and bounded, then $\nu^*_{\mathsf{R}} = \nu^*_{\mathsf{R+}}$.    
\end{corollary}
\begin{proof}
If $\mathsf{F}$ is nonempty and bounded, then the feasible region of \eqref{rlt} is nonempty and bounded~\cite[Lemma 7]{QiuY23a}, which implies that $\nu^*_{\mathsf{R}}$ is finite. The assertion follows from Theorem~\ref{rltcserlt}.

\end{proof}

We close this section with an example that illustrates a family of instances of \eqref{QP} with an unbounded feasible region but a finite RLT lower bound $\nu^*_{\mathsf{R}}$.

\begin{example}
\label{ubex0}
Consider the following family of instances of \eqref{QP} with $n = 2$, $m = 2$, and $p = 0$ parametrized by $\alpha \in \R$, adapted from \cite[Example 4]{QiuY23a}:
\[
Q = \begin{bmatrix}1 & 1 \\ 1 & 1 \end{bmatrix}, \quad c(\alpha) = - \begin{bmatrix} \alpha \\ \alpha \end{bmatrix}, \quad G = \begin{bmatrix} 1 & -1 \\ 1 & -1 \end{bmatrix}, \quad g = \begin{bmatrix} 2 \\ 2 \end{bmatrix}.
\]
The feasible region $\mathsf{F} = \{x \in \R^2: -2 \leq x_1 + x_2 \leq 2\}$ is clearly unbounded and the objective function is convex. It is easy to verify that the objective function can be expressed as $\mathsf{q}(x; {\alpha}) = \textstyle\frac{1}{2}\left( x_1 + x_2 - \alpha \right)^2 - \textstyle\frac{1}{2} \alpha^2$. Therefore, the set of KKT points (and hence the set of minimizers), denoted by $\mathsf{F}^{\mathsf{c}}(\alpha)$, is given by
\[
\mathsf{F}^{\mathsf{c}}(\alpha) = \left\{ \begin{array}{ll} \left\{x \in \R^2: x_1 + x_2 = -2 \right\}, & \alpha \in (-\infty, -2], \\[0.25em]
\left\{x \in \R^2: x_1 + x_2 = \alpha \right\}, & \alpha \in (-2, 2), \\[0.25em]
\left\{x \in \R^2: x_1 + x_2 = 2 \right\}, & \alpha \in [2, \infty), \\
\end{array}
\right.
\]
which implies that the optimal value, denoted by $\nu^*(\alpha)$, is given by 
\[
\nu^*(\alpha) = \left\{ \begin{array}{ll} 2 \alpha + 2, & \alpha \in (-\infty, -2], \\[0.25em]
- \textstyle\frac{1}{2} \alpha^2, & \alpha \in (-2, 2), \\[0.25em]
-2 \alpha + 2, & \alpha \in [2, \infty). \\
\end{array}
\right.
\] 
It is easy to verify that an optimal solution of the RLT relaxation \eqref{rlt}, denoted by $(\hat x (\alpha), \hat X (\alpha))$, is given by 
\[
\left( \hat x (\alpha), \hat X (\alpha) \right) = \left\{\begin{array}{ll} \left( \begin{bmatrix} 0 \\ -2 \end{bmatrix}, \begin{bmatrix} 0 & 2 \\ 2 & 0 \end{bmatrix} \right), & \alpha \in (-\infty, -2], \\[1em]
\left( \begin{bmatrix} 0 \\ 0 \end{bmatrix}, \begin{bmatrix} 0 & -2 \\ -2 & 0 \end{bmatrix} \right) , & \alpha \in (-2, 2), \\[1em]
\left( \begin{bmatrix} 0 \\ 2 \end{bmatrix}, \begin{bmatrix} 0 & 2 \\ 2 & 0 \end{bmatrix} \right), & \alpha \in [2, \infty), \\
\end{array}
\right.
\]
which implies that the optimal value of \eqref{rlt}, denoted by $\nu^*_{\mathsf{R}}(\alpha)$ is given by
\[
\nu^*_{\mathsf{R}}(\alpha) = \left\{ \begin{array}{ll} 2 \alpha + 2, & \alpha \in (-\infty, -2], \\
-2, & \alpha \in (-2, 2), \\
-2 \alpha + 2, & \alpha \in [2, \infty). \\
\end{array}
\right.
\]
Therefore, we obtain $\nu^*_{\mathsf{R}}(\alpha) = \nu^* (\alpha)$ for each $\alpha \in (-\infty,2] \cup [2,\infty)$, whereas $\nu^*_{\mathsf{R}}(\alpha) < \nu^* (\alpha)$ for each $\alpha \in (-2, 2)$. 
Since $\nu^*_{\mathsf{R}}(\alpha) > -\infty$ for each $\alpha \in \R$, 
it follows from Theorem~\ref{rltcserlt} that $\nu^*_{\mathsf{R}}(\alpha) = \nu^*_{\mathsf{R+}}(\alpha)$ for each $\alpha \in \R$. As illustrated by this example, relaxations of the first-order optimality conditions do not improve the RLT bound even though \eqref{QP} is a convex optimization problem. 
\end{example}

\subsection{Comparison of Two RLT Relaxations: Unbounded Case}

By Theorem~\ref{rltcserlt}, the RLT relaxation \eqref{erlt} is equivalent to \eqref{rlt} whenever $\nu^*_{\mathsf{R}} > -\infty$. In this section, we consider the behavior of \eqref{erlt} on instances of \eqref{QP} that admit an unbounded RLT relaxation. We present three examples that illustrate different scenarios.

\begin{example} \label{ubexparam1}
Consider the following family of instances of \eqref{QP} with $n = 2$, $m = 2$, and $p = 0$ parametrized by $\alpha \in \R$:
\[
Q = \begin{bmatrix}1 & 0 \\ 0 & 1 \end{bmatrix}, \quad c(\alpha) = -   \begin{bmatrix} \alpha \\ \alpha \end{bmatrix}, \quad G = \begin{bmatrix} 1 & -1 \\ 1 & -1 \end{bmatrix}, \quad g = \begin{bmatrix} 2 \\ 2 \end{bmatrix}.
\]
The feasible region $\mathsf{F} = \{x \in \R^2: -2 \leq x_1 + x_ 2 \leq 2\}$ is clearly unbounded and the objective function is strictly convex. It is easy to verify that the unique KKT point (and hence the unique minimizer) is given by
\[
\mathsf{F}^{\mathsf{c}}(\alpha) = \left\{  \begin{array}{ll} \left\{ \begin{bmatrix} -1 & -1 \end{bmatrix}^\top \right\}, & \alpha \in (-\infty, -1], \\[0.25em]
\left\{ \begin{bmatrix} \alpha & \alpha \end{bmatrix}^\top \right\}, & \alpha \in (-1, 1), \\[0.25em]
\left\{ \begin{bmatrix} 1 & 1 \end{bmatrix}^\top \right\}, & \alpha \in [1, \infty), \\
\end{array}
\right.
\]
which implies that 
\[
\nu^*(\alpha) = \left\{ \begin{array}{ll} 2 \alpha + 1, & \alpha \in (-\infty, -1], \\[0.25em]
-\alpha^2, & \alpha \in (-1, 1), \\[0.25em]
-2 \alpha + 1, & \alpha \in [1, \infty). \\
\end{array}
\right.
\]
For each $\alpha \in \R$, we claim that the RLT relaxation \eqref{rlt} is unbounded below. Consider the dual problem given by \eqref{rltd}. Since $p = 0$, the second constraint implies that 
\[
Q = G \, S \, G^\top = \left( S_{11} - 2 S_{12} + S_{22} \right) \begin{bmatrix} 1 & 1 \\ 1 & 1 \end{bmatrix},
\]
where $S \in \cS^2$. It follows that \eqref{rltd} is infeasible, which implies that \eqref{rlt} is unbounded below since $\mathsf{F} \neq \emptyset$. Therefore, $\nu^*_{\mathsf{R}}(\alpha) = -\infty$ for each $\alpha \in \R$. 

Let us now consider \eqref{erlt}. We remark that both the objective function and the constraints of \eqref{erlt} depend on $\alpha$. Our computational results illustrate that the $(x, X)$-component of an optimal solution of \eqref{erlt} is given by
\[
\left( x^*(\alpha), X^*(\alpha) \right) = \left\{\begin{array}{ll} \left( \begin{bmatrix} -1 \\ -1 \end{bmatrix}, \begin{bmatrix} 1 & 1 \\ 1 & 1 \end{bmatrix} \right), & \alpha \in (-\infty, -1], \\[1em]
\left( \begin{bmatrix} 0 \\ 0 \end{bmatrix}, \begin{bmatrix} -1 & -1 \\ -1 & -1 \end{bmatrix} \right) , & \alpha \in (-1, 1), \\[1em]
\left( \begin{bmatrix} 1 \\ 1 \end{bmatrix}, \begin{bmatrix} 1 & 1 \\ 1 & 1 \end{bmatrix} \right), & \alpha \in [1, \infty), \\
\end{array}
\right.
\]
which implies that
\[
\nu_{\mathsf{R+}}^*(\alpha) = \left\{ \begin{array}{ll} 2 \alpha + 1, & \alpha \in (-\infty, -1], \\[0.25em]
-1, & \alpha \in (-1, 1), \\[0.25em]
-2 \alpha + 1, & \alpha \in [1, \infty). \\
\end{array}
\right.
\]
It follows that $\nu_{\mathsf{R+}}^*(\alpha) = \nu^*(\alpha)$ for each $\alpha \in (-\infty, -1] \cup [1, \infty)$, whereas $\nu_{\mathsf{R+}}^*(\alpha) < \nu^*(\alpha)$ for each $\alpha \in (-1,1)$.
\end{example}

In Example~\ref{ubexparam1}, \eqref{erlt} not only strengthens \eqref{rlt}, but also fully closes the gap for certain choices of the parameter $\alpha$. On the other hand, one can still have $\nu_{\mathsf{R+}} < \nu^*$, i.e., the relaxation of optimality conditions does not guarantee exactness even if the objective function of \eqref{QP} is strictly convex.   

\begin{example} \label{ubex2}
Consider the following family of instances of \eqref{QP} with $n = 2$, $m = 2$, and $p = 0$ parametrized by $\alpha \in \R$:  
\[
Q = - \begin{bmatrix}1 & 0 \\ 0 & 1 \end{bmatrix}, \quad c(\alpha) =    \begin{bmatrix} \alpha \\ \alpha \end{bmatrix}, \quad G = \begin{bmatrix} 1 & -1 \\ 1 & -1 \end{bmatrix}, \quad g = \begin{bmatrix} 2 \\ 2 \end{bmatrix}.
\]
Note that this family of instances is obtained by negating the objective function of the family of instances in Example~\ref{ubexparam1}. The objective function is strictly concave and is unbounded along the direction $d = \begin{bmatrix} 1 & -1 \end{bmatrix}^\top$. Therefore, $\nu^*(\alpha) = -\infty$ for each $\alpha \in \R$. By \eqref{rlt-lb}, $\nu_{\mathsf{R}}(\alpha) = -\infty$ for each $\alpha \in \R$.

It is easy to verify that the set of KKT points is given by
\[
\mathsf{F}^{\mathsf{c}}(\alpha) = \left\{ \begin{array}{ll} \left\{ \begin{bmatrix} 1 & 1 \end{bmatrix}^\top \right\}, & \alpha \in (-\infty, -1), \\[0.25em]
\left\{ \begin{bmatrix} -1 & -1 \end{bmatrix}^\top, \begin{bmatrix} \alpha & \alpha \end{bmatrix}^\top, \begin{bmatrix} 1 & 1 \end{bmatrix}^\top \right\}, & \alpha \in [-1, 1], \\[0.25em]
\left\{ \begin{bmatrix} -1 & -1 \end{bmatrix}^\top \right\}, & \alpha \in [1, \infty), \\
\end{array}
\right.
\]

Concerning \eqref{erlt}, it follows from our computational results that the $(x, X)$-component of an optimal solution of \eqref{erlt} is given by
\[
\left( x^*(\alpha), X^*(\alpha) \right) = \left\{\begin{array}{ll} \left( \begin{bmatrix} 1 \\ 1 \end{bmatrix}, \begin{bmatrix} 1 & 1 \\ 1 & 1 \end{bmatrix} \right), & \alpha \in (-\infty, 0), \\[1em]
\left( \begin{bmatrix} 1 \\ 1 \end{bmatrix}, \begin{bmatrix} 1 & 1 \\ 1 & 1 \end{bmatrix} \right) , & \alpha = 0, \\[1em]
\left( \begin{bmatrix} -1 \\ -1 \end{bmatrix}, \begin{bmatrix} 1 & 1 \\ 1 & 1 \end{bmatrix} \right), & \alpha \in (0, \infty). \\
\end{array}
\right.
\]
Therefore,
\[
\nu_{\mathsf{R+}}^*(\alpha) = \left\{ \begin{array}{ll} 2 \alpha - 1, & \alpha \in (-\infty, 0), \\[0.25em]
-1, & \alpha = 0, \\[0.25em]
-2 \alpha - 1, & \alpha \in (0, \infty). \\
\end{array}
\right.
\]
It follows that $\nu_{\mathsf{R+}}^*(\alpha) > -\infty = \nu^*$ for each $\alpha \in \R$. 
Therefore, \eqref{erlt} is no longer a valid relaxation of \eqref{QP}.
\end{example}

The next example illustrates another possible scenario.

\begin{example} \label{ubex3}
Consider the following family of instances of \eqref{QP} with $n = 2$, $m = 2$, and $p = 0$:  
\[
Q = - \begin{bmatrix}1 & 0 \\ 0 & 1 \end{bmatrix}, \quad c(\alpha) =    \begin{bmatrix} -\alpha \\ -1 + \alpha \end{bmatrix}, \quad G = - \begin{bmatrix} 1 & 0 \\ 0 & 1 \end{bmatrix}, \quad g = \begin{bmatrix} 0 \\ 0 \end{bmatrix}.
\]
The objective function is strictly concave and is unbounded along the direction $d = \begin{bmatrix} 1 & 1 \end{bmatrix}^\top$. Therefore, $\nu^*(\alpha) = -\infty$ for each $\alpha \in \R$. By \eqref{rlt-lb}, $\nu_{\mathsf{R}}^*(\alpha) = -\infty$ for each $\alpha \in \R$.

It is easy to verify that the set of KKT points is empty, i.e., $\mathsf{F}^{\mathsf{c}}(\alpha) = \emptyset$ for each $\alpha \in \R$ since \eqref{kkt1} and \eqref{kkt5} are inconsistent. It follows that \eqref{erlt} is infeasible, i.e., $\nu_{\mathsf{R+}}^*(\alpha) = +\infty > -\infty = \nu^*$ for each $\alpha \in \R$. 
Once again, \eqref{erlt} is no longer a valid relaxation of \eqref{QP}.
\end{example}

We close this section with a full comparison of the RLT relaxations \eqref{rlt} and \eqref{erlt}.

\begin{enumerate}
    \item[(i)] If \eqref{QP} has a finite optimal value, then \eqref{eqp} is an equivalent reformulation of \eqref{QP} by Lemma~\ref{eqp_props}(iii). 
    \begin{itemize}
        \item[-] If $\nu^*_{\mathsf{R}} > -\infty$, then \eqref{rlt} and \eqref{erlt} are equivalent and $\nu^*_{\mathsf{R}} = \nu^*_{\mathsf{R+}}$ by Theorem~\ref{rltcserlt} (see, e.g., Example~\ref{ubex0}).

        \item[-] If, on the other hand, $\nu^*_{\mathsf{R}} = -\infty$, then $\nu^*_{\mathsf{R}} \leq \nu^*_{\mathsf{R+}} \leq \nu^*$, and each of the two inequalities can be strict (see, e.g.,  Example~\ref{ubexparam1}). In this case, \eqref{erlt} can potentially yield a tighter lower bound than \eqref{rlt}.
    \end{itemize}  

    \item[(ii)] If \eqref{QP} is unbounded below, then $\nu^*_{\mathsf{R}} = -\infty$ by \eqref{rlt-lb}. 
    \begin{itemize}
        \item[-] If $\mathsf{F}^{\mathsf{c}} = \emptyset$, then \eqref{eqp} is infeasible by Lemma~\ref{eqp_props}(i). Therefore, \eqref{eqp} is no longer equivalent to \eqref{QP}, and \eqref{erlt} is not necessarily a valid relaxation of \eqref{QP} (see, e.g., Example~\ref{ubex3}). 
        \item[-] If, on the other hand, $\mathsf{F}^{\mathsf{c}} \neq \emptyset$, then $\nu^*_{\mathsf{R+}} > -\infty$ by Lemma~\ref{eqp_props}(ii). Therefore, once again, \eqref{eqp} is no longer equivalent to \eqref{QP}, and \eqref{erlt} is not necessarily a valid relaxation of \eqref{QP} (see, e.g., Example~\ref{ubex2}).
        \end{itemize}
\end{enumerate}

In Examples~\ref{ubex2} and~\ref{ubex3}, it is easy to verify that $-\infty = \nu^*(\alpha) < \nu^*_{\mathsf{+}}(\alpha) = \nu^*_{\mathsf{R+}}(\alpha)$ for each $\alpha \in \R$. Therefore, for each family of instances, \eqref{erlt} is, in fact, an exact relaxation of \eqref{eqp}. However, \eqref{eqp} is not equivalent to \eqref{QP} in each of these examples. Therefore, some care needs to be taken before incorporating relaxations of optimality conditions into RLT relaxations of general quadratic programs. We refer the reader to~\cite{HuMP12} for a two-stage approach to obtain equivalent linear complementarity formulations of \eqref{QP}.

\section{SDP-RLT Relaxation} \label{Sec4}

In this section, we consider the SDP-RLT relaxation of \eqref{QP} given by combining the Shor relaxation with the RLT relaxation. After presenting the SDP-RLT relaxation, we show that the Slater condition fails whenever $p > 0$. We then present another formulation using a facial reduction approach, for which the Slater condition always holds. Using the equivalence between the two formulations, we establish optimality conditions for the original SDP-RLT relaxation. Finally, we present our main result on the effect of incorporating optimality conditions on the strength of the SDP-RLT relaxation. 

The SDP-RLT relaxation of \eqref{QP} is obtained by further strengthening the relation between the original variable $x$ and the lifted variable $X$ in the RLT relaxation using an additional semidefinite constraint:
\[
\begin{array}{llrcl}
\tag{$\mathsf{SR}$} \label{srlt} & \nu^*_{\mathsf{SR}} = \min\limits_{x \in \R^n, X \in \cS^n} & \textstyle\frac{1}{2}\langle Q, X \rangle + c^\top x & & \\
 & \textrm{s.t.} & & & \\
 & & G^\top x & \leq & g\\[0.2em]
 & & H^\top x & = & h \\[0.2em]
 & & H^\top X & = & h \, x^\top \\[0.2em]
 & & G^\top X G - G^\top x \, g^\top - g \, x^\top G + g \, g^\top & \geq & \mathbf{0}, \\[0.2em]
 & & X - x \, x^\top & \succeq & \mathbf{0}.
\end{array}
\]

By the Schur complement property,
\begin{equation} \label{Schurcomp}
X - x \, x^\top \succeq \mathbf{0} \Longleftrightarrow \begin{bmatrix}
   1 & x^\top \\
   x & X
\end{bmatrix} \succeq \mathbf{0}.    
\end{equation}
Therefore, \eqref{srlt} can be formulated as a semidefinite programming problem. 

Similar to the RLT relaxation \eqref{rlt}, for each $x \in \mathsf{F}$, $(x, x \, x^\top)$ is \eqref{srlt}-feasible with the same objective function value, where $\mathsf{F}$ is given by \eqref{def_F}. Therefore, we immediately obtain
\begin{equation} \label{srltrlt-lb}
    \nu^*_{\mathsf{R}} \leq \nu^*_{\mathsf{SR}} \leq \nu^*.
\end{equation}

\subsection{Slater Condition}

In contrast with the RLT relaxation \eqref{rlt}, strong duality, in general, may fail for the SDP-RLT relaxation \eqref{srlt}. A sufficient condition that ensures strong duality is the Slater condition, i.e., the existence of a strictly feasible solution.  

First, we observe that the Slater condition fails for \eqref{srlt} whenever $p > 0$.

\begin{lemma} \label{Slater}
Given an instance of \eqref{QP} satisfying Assumption~\ref{assump1}, the Slater condition holds for the SDP-RLT relaxation \eqref{srlt} if and only if $p = 0$.    
\end{lemma}
\begin{proof}
Suppose that $p > 0$. For any \eqref{srlt}-feasible solution $(\hat x, \hat X) \in \R^n \times \cS^n$, 
\begin{equation} \label{nullsp}
\begin{bmatrix}
   1 & \hat x^\top \\
   \hat x & \hat X
\end{bmatrix} \begin{bmatrix} h^\top \\ - H \end{bmatrix} = \begin{bmatrix}
    \mathbf{0} \\ \mathbf{0}
\end{bmatrix},
\end{equation}
where we used $H^\top \hat x = h$ and $H^\top \hat X = h \, \hat x^\top$. Therefore, 
\[
\left\langle \begin{bmatrix} h^\top \\ - H \end{bmatrix} \begin{bmatrix} h^\top \\ - H \end{bmatrix}^\top , \begin{bmatrix}
   1 & \hat x^\top \\
   \hat x & \hat X
\end{bmatrix} \right\rangle = 0,
\]
which implies that \eqref{srlt} cannot have a strictly feasible solution.

For $p = 0$, let $\hat x = x^\circ$, where $x^\circ$ is defined as in Assumption~\ref{assump1}, and let $\hat X = \hat x \, \hat x^\top + \epsilon \, I$, where $\epsilon \in \R_{++}$. Therefore,
\[
G^\top \hat X G - G^\top \hat x \, g^\top - g \, \hat x^\top G + g \, g^\top = \left( g - G^\top \hat x \right) \left( g - G^\top \hat x \right)^\top + \epsilon \, G^\top G.
\]
Since $\left( g - G^\top \hat x \right) \left( g - G^\top \hat x \right)^\top > \mathbf{0}$ by Assumption~\ref{assump1}, it follows that there exists a sufficiently small $\epsilon > 0$ such that both the linear inequality and semidefinite constraints in \eqref{srlt} are satisfied strictly. Therefore, the Slater condition holds for \eqref{srlt}. 

\end{proof}

\subsection{An Equivalent Reformulation via Facial Reduction}

By Lemma~\ref{Slater}, the Slater condition fails for the SDP-RLT relaxation \eqref{srlt} whenever $p > 0$. In view of this result, we present an equivalent reformulation of \eqref{srlt} that relies on facial reduction techniques aimed at reformulating the original problem by restricting it to the smallest face of the positive semidefinite cone that contains the feasible region (see, e.g.,~\cite{BW1981a,BW1981b,Pat2013,WakiM13}). 

Consider an instance of \eqref{QP} that satisfies Assumption~\ref{assump1}. If $p > 0$, then let $P$ be a matrix whose columns form a basis for the null space of $H^\top \in \R^{p \times n}$, i.e.,
\begin{equation} \label{defMmat}
H^\top P = \mathbf{0}, \quad P \in \R^{n \times (n - p)}, \quad \text{and} \quad \rank(P) = n - p,   \quad \text{if}~p \in \{1,\ldots,n-1\}. 
\end{equation}

If, on the other hand, $p = 0$, we then define 
\begin{equation} \label{defMmat0}
P = I \in \cS^n, \quad \text{if}~p = 0,   
\end{equation}

Next, let 
\begin{equation} \label{defU}
U := \begin{bmatrix} 1 & \mathbf{0} \\ x^\circ & P \end{bmatrix} \in \R^{(n+1) \times (1 + n - p)},    
\end{equation}
where $x^\circ$ is as defined in Assumption~\ref{assump1}. By \eqref{defMmat}, \eqref{defMmat0}, and \eqref{defU}, $\rank(U) = n - p + 1$.

Consider the following optimization problem:
\[
\begin{array}{llrcl}
\tag{$\mathsf{SRR}$} \label{srltr} & \nu^*_{\mathsf{SRR}} = \min\limits_{x \in \R^n, X \in \cS^n} & \textstyle\frac{1}{2}\langle Q, X \rangle + c^\top x & & \\
 & \textrm{s.t.} & & & \\
 & & G^\top x & \leq & g\\[0.2em]
 & & G^\top X G - G^\top x \, g^\top - g \, x^\top G + g \, g^\top & \geq & \mathbf{0}, \\[0.2em]
 & & \begin{bmatrix} 1 & x^T \\
 x & X \end{bmatrix} & \in & \cK,
\end{array}
\]
where 
\begin{equation} \label{defK}
\cK := \left\{ K \in \cS^{n+1}: K = U \, T \, U^\top, \quad T \in \cS^{n - p + 1}_+ \right\} \subseteq \cS^{n+1}_+,    
\end{equation}
and $U$ is given by \eqref{defU}. It is well-known and easy to check that $\cK$ is a face of the cone of positive semidefinite matrices (see, e.g.,~\cite{Pat2013}) and that it is a closed convex cone~\cite[Corollary 18.1.1]{rockafellar1970convex}. Furthermore, it is easy to verify that the dual cone is given by
\begin{equation} \label{defK*}
 \cK^* = \left\{L \in \cS^{n+1}: U^\top L \, U \succeq \mathbf{0}\right\} \supseteq \cS^{n+1}_+.  
\end{equation}

We first review some useful properties of $\cK$ and $\cK^*$.

\begin{lemma} \label{propK}
Let $\cK$ and $\cK^*$ be given by \eqref{defK} and \eqref{defK*}, respectively. Then, 
    \begin{align}
     \cK & = \cS^{n+1}_+, \quad \text{if}~p = 0, \label{Kchar1} \\
     \cK & = \left\{ K \in \cS^{n+1}: U^\top K \, U \succeq \mathbf{0}, \quad U^\top K \, V = \mathbf{0}, \quad V^\top K \, V = \mathbf{0} \right\}, \quad \text{if}~p > 0, \label{Kchar2} \\
     \cK^* & = \cS^{n+1}_+, \quad \text{if}~p = 0, \label{K*char1} \\
     \cK^* & = \left\{L \in \cS^{n+1}: \begin{array}{rcl} L & = & \begin{bmatrix} U & V \end{bmatrix} \begin{bmatrix} N_{UU} & N_{UV} \\[0.2em] N_{UV}^\top & N_{VV} \end{bmatrix} \begin{bmatrix} U & V \end{bmatrix}^\top, \\[0.2em]
     N_{UU} & \in & \cS^{n + 1 - p}_+, \\[0.2em]
     N_{UV} & \in & \R^{(n + 1 - p) \times p}, \\[0.2em]
     N_{VV} & \in & \cS^p
     \end{array} \right\}, \quad \text{if}~p > 0,\label{K*char2}
    \end{align}
    where $V \in \R^{(n+1) \times p}$ is a matrix such that $\begin{bmatrix} U & V \end{bmatrix} \in \R^{(n+1) \times (n+1)}$ is invertible and $U^\top V = \mathbf{0} \in \R^{(n- p +1) \times p}$. Furthermore, 
    \begin{align}
     \relint(\cK) & = \intK(\cK) = \cS^{n+1}_{++}, \quad \text{if}~p = 0, \label{relint1} \\ 
      \relint(\cK) & = \left\{ K \in \cS^{n+1}: K = U \, T \, U^\top, \quad T \in \cS^{n - p + 1}_{++} \right\}, \quad \text{if}~p > 0. 
    \end{align}   
\end{lemma}
\begin{proof}
The reader is referred to \cite[Lemma 1]{PermenterP18} and \cite{DrusvyatskiyW17}.   

\end{proof}

We next establish the equivalence between \eqref{srlt} and \eqref{srltr}.

\begin{proposition} \label{srvssrr} \label{srltvssrltr}
The optimization problems \eqref{srlt} and \eqref{srltr} are equivalent, and $\nu^*_{\mathsf{SR}} = \nu^*_{\mathsf{SRR}}$.
\end{proposition}
\begin{proof}
First, assume that $p = 0$. If $(\hat x, \hat X) \in \R^n \times \cS^n$ is \eqref{srltr}-feasible, then $(\hat x, \hat X) \in \R^n \times \cS^n$ is \eqref{srlt}-feasible since $\cK = \cS^n_+$ by Lemma~\ref{propK}. 

Conversely, if $(\hat x, \hat X) \in \R^n \times \cS^n$ is \eqref{srlt}-feasible, then
\[
\begin{bmatrix}
   1 & \hat x^\top \\
   \hat x & \hat X
\end{bmatrix} = U \left( U^{-1} \begin{bmatrix}
   1 & \hat x^\top \\
   \hat x & \hat X
\end{bmatrix} \left( U^{-1} \right)^\top \right) U^\top = U \, T \, U^\top \in \cK,
\]
where we used the fact that $T \succeq \mathbf{0}$ and $U \in \R^{(n+1) \times (n+1)}$ is invertible by \eqref{defU} and \eqref{defMmat0}. Therefore, $(\hat x, \hat X) \in \R^n \times \cS^n$ is \eqref{srltr}-feasible.

Suppose now that $p > 0$. We define 
\begin{equation} \label{defT}
V = \begin{bmatrix} h^\top \\ - H \end{bmatrix} \in \R^{(n + 1) \times p}.
\end{equation}
Note that 
\[
U^\top V = \begin{bmatrix}
    1 &  (x^\circ)^\top \\ \mathbf{0} & P^\top 
\end{bmatrix} \begin{bmatrix} h^\top \\ - H \end{bmatrix} = \begin{bmatrix}
    \mathbf{0} \\ \mathbf{0}
\end{bmatrix},
\]
where we used Assumption~\ref{assump1} and \eqref{defMmat}. Furthermore, for $(\alpha, a) \in \R \times \R^n$, consider the following system:
\begin{equation} \label{sys1}
\begin{bmatrix} U & V \end{bmatrix}^\top \begin{bmatrix} \alpha \\ a \end{bmatrix} = \begin{bmatrix} 1 &  (x^\circ)^\top \\ \mathbf{0} & P^\top \\
h & -H^\top
\end{bmatrix} \begin{bmatrix} \alpha \\ a \end{bmatrix} =
\begin{bmatrix}
  0 \\ \mathbf{0} \\ \mathbf{0}  
\end{bmatrix}.
\end{equation}
Since $H \in \R^{n \times p}$ has full column rank by Assumption~\ref{assump1}, it follows from  \eqref{defMmat} that the columns of $H$ form a basis for the null space of $P^\top$. By the second row of \eqref{sys1}, we have $P^\top a = \mathbf{0}$, which implies that there exists a unique vector $d \in \R^p$ such that $H \, d = a$. Substituting this into the first row of \eqref{sys1} and using Assumption~\ref{assump1}, we obtain
\[
\alpha + (x^\circ)^\top a = \alpha + (x^\circ)^\top H \, d = \alpha + h^\top d = 0.
\]
Finally, multiplying the third row of \eqref{sys1} from the left by $-d^\top$, using the above identity and $H \, d = a$, we have
\[
- \alpha \, d^\top h + d^\top H^\top a = \alpha^2 + a^\top a = 0,
\]
which implies that the unique solution of \eqref{sys1} is given by $(\alpha, a) = (0, \mathbf{0}) \in \R \times \R^n$. Therefore, $\begin{bmatrix} U & V \end{bmatrix} \in \R^{(n + 1) \times (n + 1)}$ is an invertible matrix that satisfies the assumptions of Lemma~\ref{propK}.

If $(\hat x, \hat X) \in \R^n \times \cS^n$ is \eqref{srltr}-feasible, then $\hat X - \hat x \, \hat x \succeq \mathbf{0}$ by \eqref{defK} and \eqref{Schurcomp}. Therefore, there exists a $\hat D \in \R^{n \times n}$ such that 
$\hat X = \hat x \, \hat x^\top + \hat D \, \hat D^\top$. By \eqref{Kchar2},
\begin{eqnarray*}
V^\top \begin{bmatrix}
   1 & \hat x^\top \\
   \hat x & \hat X
\end{bmatrix} V & = & \begin{bmatrix} h & - H^\top \end{bmatrix} \begin{bmatrix}
   1 & \quad \hat x^\top \\
   \hat x & \quad \hat x \, \hat x^\top + \hat D \, \hat D^\top
\end{bmatrix} \begin{bmatrix} h^\top \\ - H \end{bmatrix} \\
 & = & \left(h - H^\top \hat x \right) \left(h - H^\top \hat x \right)^\top + \left( H^\top \hat D \right) \left( H^\top \hat D \right)^\top \\
 & = & \mathbf{0},
\end{eqnarray*}
which implies that $H^\top \hat x = h$ and $H^\top \hat D = \mathbf{0}$. Therefore, 
\[
H^\top \hat X = H^\top \hat x \, \hat x^\top + H^\top \hat D \, \hat D^\top = h \, \hat x^\top,
\]
which implies that $(\hat x, \hat X) \in \R^n \times \cS^n$ is \eqref{srlt}-feasible. 

Conversely, if $(\hat x, \hat X) \in \R^n \times \cS^n$ is \eqref{srlt}-feasible, then $(\hat x, \hat X) \in \R^n \times \cS^n$ is \eqref{srltr}-feasible by \eqref{Schurcomp}, \eqref{nullsp}, and \eqref{Kchar2}. 

The last assertion follows from the fact that the objective functions are the same.

\end{proof}

\subsection{Optimality Conditions}

In this section, we first present the dual problem of the equivalent formulation \eqref{srltr}. By relying on Proposition~\ref{srvssrr}, we establish necessary and sufficient optimality conditions for \eqref{srlt}. Finally, we deduce some useful relations about the set of optimal solutions of \eqref{srlt}.

First, we consider the dual problem of the reduced SDP-RLT relaxation \eqref{srltr}. We define the dual variables $(u, S) \in \R^m \times \cS^m$ corresponding to the two sets of linear constraints in \eqref{srltr}. Concerning the conic constraint, we introduce the dual variable
\[
\begin{bmatrix}
    \beta & b^\top \\
    b & B 
\end{bmatrix} \in \cS^{n + 1}, 
\]
where $\beta \in \R$, $b \in \R^{n}$, and $B \in \cS^{n}$. After scaling the above dual matrix variable and $S$ by a factor of $\textstyle\frac{1}{2}$, we obtain the following dual problem of \eqref{srltr}:

\[
\begin{array}{llrcl}
\tag{$\mathsf{SRRD}$} \label{srltrd} & \nu^*_{\mathsf{SRRD}} 
 = \max\limits_{\substack{u \in \R^m, S \in \cS^m \\
\beta \in \R, b \in \R^{n}, B \in \cS^{n}}} & -u^\top g - \textstyle\frac{1}{2} g^\top S \, g - \textstyle\frac{1}{2} \beta & & \\
 & \textrm{s.t.} & & & \\
  & & - G \, u - G \, S \, g + b & = & c \\[0.25em]
 & &  G \, S \, G^\top + B & = & Q \\[0.25em]
 & & S & \geq & \mathbf{0} \\[0.25em]
 & & u & \geq & \mathbf{0} \\[0.25em]
 & & \begin{bmatrix}
    \beta & b^\top \\
    b & B 
\end{bmatrix} & \in & \cK^*,
\end{array}
\]
where $\cK^*$ is defined as in \eqref{defK*}.

We next present the optimality conditions for the SDP-RLT formulation \eqref{srlt}.

\begin{proposition} \label{opt_cond_srr}
Let $(\hat x, \hat X) \in \R^{n} \times \cS^{n}$ be \eqref{srlt}-feasible. Then, $(\hat x, \hat X)$ is an optimal solution of \eqref{srlt} if and only if 
there exists 
$(\hat u, \hat S, \hat \beta, \hat b, \hat B) \in \R^m \times \cS^m \times \R \times \R^{n} \times \cS^{n}$ such that
\begin{eqnarray} \label{opt_form_srr}
- G \, \hat u - G \, \hat S \, g + \hat b & = & c, \label{opt_c_srr} \\
G \, \hat S \, G^\top + \hat B & = & Q, \label{opt_Q_srr} \\ 
\hat u & \geq & \mathbf{0}, \label{opt_u_srr} \\
\hat S & \geq & \mathbf{0}, \label{opt_S_srr} \\
\begin{bmatrix}
    \hat \beta & \hat b^\top \\
    \hat b & \hat B 
\end{bmatrix} & \in & \cK^*,  \label{opt_B_b_srr} \\
\hat u^\top \left( g - G^\top \hat x \right) & = & 0, \label{opt_cs1_srr} \\
\left\langle \hat S, G^\top \hat X \, G - G^\top \hat x \, g^\top - g \,  \hat x^\top G + g \, g^\top \right\rangle & = & 0, \label{opt_cs2_srr} \\
\left\langle \begin{bmatrix}
    \hat \beta & \hat b^\top \\ \hat b & \hat B 
\end{bmatrix}, \begin{bmatrix}
    1 & \hat x^\top \\ \hat x & \hat X 
\end{bmatrix} \right\rangle & = & 0. \label{opt_sc3_srr}
\end{eqnarray}
Furthermore, $\nu^*_{\mathsf{SR}} = \nu^*_{\mathsf{SRR}} = \nu^*_{\mathsf{SRRD}}$.
\end{proposition}
\begin{proof}
Let $(\hat x, \hat X)$ be \eqref{srlt}-feasible. By Proposition~\ref{srvssrr}, $(\hat x, \hat X)$ is \eqref{srltr}-feasible. First, we observe that \eqref{srltr} satisfies the Slater condition. Let $\epsilon > 0$ and 
\[
\tilde T = \begin{bmatrix} 1 & \mathbf{0} \\
\mathbf{0} & \epsilon \, I \end{bmatrix} \in \cS^{n - p + 1}_{++}.
\]
By \eqref{defU} and Lemma~\ref{propK},
\[
\begin{bmatrix} 1 & \tilde x^\top \\
\tilde x & \tilde X \end{bmatrix} = U \, \tilde T \, U^\top = \begin{bmatrix} 1 & \quad \left( x^\circ \right)^\top \\ x^\circ & \quad x^\circ \, \left( x^\circ \right)^\top + \epsilon \, P \, P^\top \end{bmatrix} \in \relint(\cK).
\]
Using a similar argument as in the proof of Lemma~\ref{Slater}, there exists a sufficiently small $\epsilon > 0$ such that the linear inequality constraints of \eqref{srltr} are strictly satisfied. Therefore, the Slater condition holds for \eqref{srltr}, which implies that strong duality holds and the optimal value is attained in \eqref{srltrd}. Together with Proposition~\ref{srvssrr}, we obtain $\nu^*_{\mathsf{SR}} = \nu^*_{\mathsf{SRR}} = \nu^*_{\mathsf{SRRD}}$. The relations \eqref{opt_c_srr}--\eqref{opt_sc3_srr} now follow from standard conic duality.

\end{proof}

By relying on Proposition~\ref{opt_cond_srr}, we next deduce further useful conditions between optimal solutions of \eqref{srlt} and \eqref{srltrd}, which will play a key role in establishing our main result in the next section.

\begin{proposition} \label{useful_opt_cond_srlt}
Let $(\hat x, \hat X) \in \R^n \times \cS^n$ and $(\hat u, \hat S, \hat \beta, \hat b, \hat B) \in \R^m \times \cS^m \times \R \times \R^{n} \times \cS^{n}$ be an optimal solution of 
\eqref{srlt} and \eqref{srltr}, respectively. Then, there exists 
$(\hat z, \hat W) \in \R^p \times \R^{p \times n}$ such that
\begin{eqnarray} \label{opt_form_srr1}
- \hat b^\top \hat x + h^\top \hat z & = & \hat \beta, \label{rel_betar} \\
\hat B \, \hat x + \hat b + H \, \hat z & = & \mathbf{0}, \label{rel_zr} \\
\hat B \, \hat X + \hat b \, \hat x^\top & = & H \, \hat W. \label{rel_Wr}
\end{eqnarray}  
Furthermore,
\begin{equation}
\left\langle \hat B, \hat X - \hat x \, \hat x^\top \right\rangle = 0. \label{rel_csr} 
\end{equation}   
\end{proposition}
\begin{proof}
Let $(\hat x, \hat X) \in \R^n \times \cS^n$ and $(\hat u, \hat S, \hat \beta, \hat b, \hat B) \in \R^m \times \cS^m \times \R \times \R^{n} \times \cS^{n}$ be an optimal solution of 
\eqref{srlt} and \eqref{srltr}, respectively. By Proposition~\ref{opt_cond_srr}, the relations \eqref{opt_c_srr}--\eqref{opt_sc3_srr} hold. 
Since $(\hat x, \hat X)$ is \eqref{srltr}-feasible by Proposition~\ref{srvssrr}, there exists $\hat T \in \cS^{n - p + 1}_+$ such that
\begin{equation} \label{pr9rel1}
\begin{bmatrix}
    1 & \hat x^\top \\ \hat x & \hat X \end{bmatrix} = U \, \hat T \,  U^\top.
\end{equation}
Combining this identity with \eqref{opt_sc3_srr}, we get
\begin{equation} \label{pr9rel2}
  \left\langle \begin{bmatrix}
    \hat \beta & \hat b^\top \\ \hat b & \hat B 
\end{bmatrix}, \begin{bmatrix}
    1 & \hat x^\top \\ \hat x & \hat X 
\end{bmatrix} \right\rangle =  \left\langle \begin{bmatrix}
    \hat \beta & \hat b^\top \\ \hat b & \hat B 
\end{bmatrix}, U \, \hat T \,  U^\top \right\rangle = \left\langle U^\top \begin{bmatrix}
    \hat \beta & \hat b^\top \\ \hat b & \hat B 
\end{bmatrix} U , \hat T   \right\rangle = 
0.  
\end{equation}
By \eqref{opt_B_b_srr} and \eqref{defK*}, 
\begin{equation} \label{pr9rel3}
U^\top \begin{bmatrix}
    \hat \beta & \hat b^\top \\ \hat b & \hat B 
\end{bmatrix} U \succeq \mathbf{0}.
\end{equation}
Since $\hat T \succeq \mathbf{0}$, it follows from \eqref{pr9rel3} and \eqref{pr9rel2} that
\begin{equation} \label{pr9rel4}
U^\top \begin{bmatrix}
    \hat \beta & \hat b^\top \\ \hat b & \hat B 
\end{bmatrix} U \, \hat T = \mathbf{0}.
\end{equation}
Multiplying \eqref{pr9rel4} from the right by $U^\top$ and using \eqref{pr9rel1} and \eqref{defU}, we obtain
\[
\begin{bmatrix}
    1 &  (x^\circ)^\top \\[0.2em] \mathbf{0} & P^\top 
\end{bmatrix} \begin{bmatrix}
    \hat \beta & \hat b^\top \\[0.2em] \hat b & \hat B 
\end{bmatrix} \begin{bmatrix}
    1 & \hat x^\top \\[0.2em] \hat x & \hat X 
\end{bmatrix} = \begin{bmatrix} 0 & \mathbf{0} \\ \mathbf{0} & \mathbf{0} \end{bmatrix}.
\]
Considering the $(1,1)$-, $(2,1)$-, and $(2,2)$-blocks, we arrive at
\begin{eqnarray} \label{pr9relseq1}
\hat \beta + \hat b^\top \hat x + (x^\circ)^\top \hat b + (x^\circ)^\top B \, \hat x & = & 0, \label{pr9rel5} \\
 P^\top \left( \hat B \, \hat x + \hat b  \right) & = & \mathbf{0}, \label{pr9rel6} \\
 P^\top \left( \hat B \, \hat X + \hat b \, \hat x^\top \right) & = & \mathbf{0}. \label{pr9rel7}
\end{eqnarray}

If $p = 0$, then $P = I$ by \eqref{defMmat0}, which implies that the relations \eqref{rel_betar}, \eqref{rel_zr}, and \eqref{rel_Wr} directly follow from \eqref{pr9rel5}, \eqref{pr9rel6}, \eqref{pr9rel7}, and our conventions for empty matrices and vectors. Furthermore, 
\[
\left\langle \hat B, \hat X - \hat x \, \hat x^\top \right\rangle = \trace \left( \hat B \, \hat X - \hat B \, \hat x \, \hat x^\top \right) = \trace \left( - \hat b \, \hat x^\top + \hat b \, x^\top \right) = 0,
\]
where we used \eqref{rel_Wr} and \eqref{rel_zr} in the penultimate equality. This establishes \eqref{rel_csr}.

Let us now assume that $p > 0$. Since the columns of $H$ form a basis for the null space of $P^\top$, it follows from \eqref{pr9rel6} and  \eqref{pr9rel7} that there exists a unique $(\hat z, \hat W) \in \R^p \times \R^{p \times n}$ such that \eqref{rel_zr} and \eqref{rel_Wr} are satisfied. Substituting \eqref{rel_zr} into \eqref{pr9rel5} and using $H^T x^\circ = h$, we obtain \eqref{rel_betar}. 

Finally, 
\begin{eqnarray*}
0 & \stackrel{\eqref{opt_sc3_srr}}{=} & \left\langle \hat B, \hat X \right\rangle + 2 \, \hat b^\top \hat x + \hat \beta \\
 & \stackrel{\eqref{rel_betar}}{=} & \left\langle \hat B, \hat X \right\rangle + \, \hat b^\top \hat x + h^\top \hat z \\
 & \stackrel{\eqref{rel_zr}}{=} & \left\langle \hat B, \hat X \right\rangle - \hat x^\top \hat B \, \hat x \\
 & = & \left\langle \hat B, \hat X - \hat x \, \hat x^\top \right\rangle,
\end{eqnarray*}
which establishes \eqref{rel_csr} and completes the proof.

\end{proof}

\subsection{SDP-RLT Relaxation of the Complementarity Formulation}

Similar to \eqref{erlt}, we next consider the SDP-RLT relaxation of the complementarity formulation \eqref{eqp} obtained by incorporating the first-order optimality conditions of \eqref{QP}.

\[
\begin{array}{llrcl}
\tag{$\mathsf{SR+}$} \label{esrlt} & \nu^*_{\mathsf{SR+}} = \min\limits_{\substack{x \in \R^n, y \in \R^m, z \in \R^p, \\ X \in \cS^n, Y \in \cS^m, Z \in \cS^p, \\ M_{xy} \in \R^{n \times m}, M_{xz} \in \R^{n \times p}, M_{yz} \in \R^{m \times p}}} & \textstyle\frac{1}{2}\langle Q, X \rangle + c^\top x & & \\
 & \textrm{s.t.} & & & \\
  & & G^\top x & \leq & g\\[0.2em]
 & & H^\top x & = & h \\[0.2em]
  & & H^\top X & = & h \, x^\top \\[0.2em]
   & & G^\top X G - G^\top x \, g^\top - g \, x^\top G + g \, g^\top & \geq & \mathbf{0} \\[0.2em]
 & & Q \, x + c + G \, y + H \, z & = & \mathbf{0}\\[0.2em]
  & & H^\top M_{xy} & = & h \, y^\top \\[0.2em]
  & & H^\top M_{xz} & = & h \, z^\top \\[0.2em]
 & & Q \, X + c \, x^\top + G \, M_{xy}^\top + H \, M_{xz}^\top & = & \mathbf{0} \\[0.2em]
 & & Q \, M_{xy} + c \, y^\top + G \, Y + H \, M_{yz}^\top & = & \mathbf{0} \\[0.2em]
 & & Q \, M_{xz} + c \, z^\top + G \, M_{yz} + H \, Z  & = & \mathbf{0}\\[0.2em]
 & & \diag \left( y \, g^\top - M_{xy}^\top \, G \right) & = & \mathbf{0}\\[0.25em]
 & & y \, g^\top - M_{xy}^\top \, G & \geq & \mathbf{0} \\[0.2em]
 & & y & \geq & \mathbf{0} \\[0.2em]
 & & Y & \geq & \mathbf{0} \\[0.2em]
 & & \langle Q, X \rangle + c^\top x + g^\top y + h^\top z & = & 0 \\[0.2em]
 & & \begin{bmatrix}
 X & M_{xy} & M_{xz} \\[0.2em]
M_{xy}^\top & Y & M_{yz} \\[0.2em]
 M_{xz}^\top & M_{yz}^\top & Z 
 \end{bmatrix} - \begin{bmatrix} x \\ y \\ z \end{bmatrix} \begin{bmatrix} x \\ y \\ z \end{bmatrix}^\top & \succeq & \mathbf{0}.
\end{array}
\]

The next proposition, which is the counterpart of Theorem~\ref{rltcserlt}, establishes our main result in this section.

\begin{theorem} \label{srltcsesrlt}
Suppose that $\nu^*_{\mathsf{SR}} > - \infty$ and \eqref{srlt} has a nonempty set of optimal solutions. Then, the SDP-RLT relaxations \eqref{srlt} and \eqref{esrlt} are equivalent, i.e., $\nu^*_{\mathsf{SR}} = \nu^*_{\mathsf{SR+}}$, and \eqref{esrlt} has a nonempty set of optimal solutions. 
\end{theorem}
\begin{proof}
Our proof is similar to that of Theorem~\ref{rltcserlt}. 

Clearly, $\nu^*_{\mathsf{SR}} \leq \nu^*_{\mathsf{SR+}}$ due to the additional constraints in \eqref{esrlt}. 

For the reverse inequality, suppose that $\nu^*_{\mathsf{SR}} > -\infty$ and let $(\hat x, \hat X) \in \R^n \times \cS^n$ be an optimal solution of \eqref{srlt}. We will construct $\hat y \in \R^m, \hat z \in \R^p, \hat Y \in \cS^m, \hat Z \in \cS^p, \hat M_{xy} \in \R^{n \times m}, \hat M_{xz} \in \R^{n \times p}$, and $\hat M_{yz} \in \R^{m \times p}$ such that $(\hat x, \hat y, \hat z, \hat X, \hat Y, \hat Z, \hat M_{xy}, \hat M_{xz}, \hat M_{yz})$ is \eqref{esrlt}-feasible.

By Propositions~\ref{opt_cond_srr} and~\ref{useful_opt_cond_srlt}, there exists 
$(\hat u, \hat S, \hat \beta, \hat b, \hat B, \hat z, \hat W) \in \R^m \times \cS^m \times \R \times \R^{n} \times \cS^{n} \times \R^p \times \R^{p \times n}$ such that \eqref{opt_c_srr}--\eqref{opt_sc3_srr} and 
\eqref{rel_betar}--\eqref{rel_Wr} are satisfied.

First, assume that $p > 0$. 
By \eqref{rel_zr}, we have $H \, \hat z = -(\hat B \, \hat x + \hat b)$.
Since $H \in \R^{n \times p}$ has full column rank, $\hat z \in \R^p$ is uniquely defined by 
\begin{equation} \label{minnorm_z}
\hat z = - (H^\top H)^{-1} H^\top \left(\hat B \, \hat x + \hat b \right).
\end{equation}
If, on the other hand, $p = 0$, then it follows from \eqref{emptyinv} that $\hat z$ given by \eqref{minnorm_z} is a well-defined $0 \times 1$ empty vector. Therefore, we will continue to use \eqref{minnorm_z} in the remainder of the proof without distinguishing between the cases of $p = 0$ and $p > 0$. 

Combining \eqref{rel_zr} with \eqref{opt_Q_srr} and \eqref{opt_c_srr}, we obtain
\[
Q \, \hat x + c + G \left(\hat S \left( g - G^\top \hat x \right) + \hat u \right) + H \, \hat z = \mathbf{0}.
\]

Therefore, in view of the fifth constraint of \eqref{esrlt}, we similarly define
\begin{equation} \label{esrlt-opty}
\hat y = \hat S \left( g - G^\top \hat x \right) + \hat u,
\end{equation}
and $\hat z$ as in \eqref{minnorm_z}.

Similar to the proof of Theorem~\ref{rltcserlt}, we define each of the remaining lifted variables by using the corresponding rank-one matrix, except that we replace $\hat x \, \hat x^\top$ by $\hat X$:
\begin{align} 
\hat Y = & \, \, \hat S \left( G^\top \hat X G - G^\top \hat x \, g^\top - g \, \hat x^\top G + g \, g^\top \right) \hat S \nonumber \\ 
 & \quad + \hat S \left( g - G^\top \hat x \right) \hat u^\top + \hat u \left(g - G^\top \hat x \right)^\top \hat S + \hat u \, \hat u^\top, \label{esrlt-optY} \\
\hat Z = & \, \, (H^\top H)^{-1} H^\top \left( \hat B \, \hat X \hat B^\top + \hat B \, \hat x \, \hat b^\top + \hat b \, \hat x^\top \hat B^\top + \hat b \, \hat b^\top \right) H (H^\top H)^{-1}, \label{esrlt-optZ} \\
\hat M_{xy} = & \, \, \hat x \, g^\top \hat S - \hat X G \, \hat S + \hat x \, \hat u^\top, \label{esrlt-optU} \\
\hat M_{xz} = & - \hat X \hat B^\top H (H^\top H)^{-1} - \hat x \, \hat b^\top H (H^\top H)^{-1}, \label{esrlt-optV} \\
\hat M_{yz} = & - \hat S \, g \, \hat x^\top \hat B^\top H (H^\top H)^{-1} - \hat S \, g \, \hat b^\top H (H^\top H)^{-1}
+ \hat S \, G^\top \hat X \hat B^\top H (H^\top H)^{-1} \nonumber \\
 & \quad + \hat S \, G^\top \hat x \, \hat b^\top H (H^\top H)^{-1} 
- \hat u \, \hat x^\top \hat B^\top H (H^\top H)^{-1} - \hat u \, \hat b^\top H (H^\top H)^{-1}. \label{esrlt-optW} 
\end{align}

We claim that $(\hat x, \hat y, \hat z, \hat X, \hat Y, \hat Z, \hat M_{xy}, \hat M_{xz}, \hat M_{yz})$ is \eqref{esrlt}-feasible. The first five constraints of \eqref{esrlt} are clearly satisfied. 

For the sixth and seventh constraints, we have
\begin{eqnarray*}
H^\top \hat M_{xy} & \stackrel{\eqref{esrlt-optU}}{=} & h \, g^\top \hat S - h \, \hat x^\top \, G \, \hat S + h \, \hat u^\top \stackrel{\eqref{esrlt-opty}}{=} h \, \hat y^\top, \\
H^\top \hat M_{xz} & \stackrel{\eqref{esrlt-optV}}{=} & 
- h \, \hat x^\top \hat B^\top H (H^\top H)^{-1} - h \, \hat b^\top H (H^\top H)^{-1}, \\
 & = & - h \, \left( \hat B \, \hat x + \hat b \right)^\top H (H^\top H)^{-1}, \\
 & \stackrel{\eqref{rel_zr}}{=} & h \, \hat z^\top, 
\end{eqnarray*}
where we used $H^\top \hat x = h$ and $H^\top \hat X = h \, \hat x^\top$.

Let us next consider the eighth constraint:
\begin{eqnarray*}
Q \, \hat X & \stackrel{\eqref{opt_Q_srr}}{=} & G \, \hat S \, G^\top \hat X + \hat B \, \hat X, \\
c \, \hat x^\top & \stackrel{\eqref{opt_c_srr}}{=} & -G \, \hat u \, \hat x^\top - G \, \hat S \, g \, \hat x^\top + \hat b \, \hat x^\top, \\ 
G \, \hat M_{xy}^\top & \stackrel{\eqref{esrlt-optU}}{=} & G \, \hat S \, g \, \hat x^\top - G \, \hat S \, G^\top \hat X + G \, \hat u \, \hat x^\top , \\  
H \, \hat M_{xz}^\top & \stackrel{\eqref{esrlt-optV}}{=} & - H (H^\top H)^{-1} H^\top \left(
\hat B \, \hat X + \hat b \, \hat x^\top \right) \stackrel{\eqref{rel_Wr}}{=} - \hat B \, \hat X - \hat b \, \hat x^\top.
\end{eqnarray*}
Therefore, 
$Q \, \hat X + c \, \hat x^\top + G \, \hat M_{xy}^\top + H \, \hat M_{xz}^\top = \mathbf{0}$.

Next, we focus on the ninth constraint:
\begin{eqnarray*}
Q \, \hat M_{xy} & \stackrel{\eqref{opt_Q_srr}, \eqref{esrlt-optU}}{=} & G \, \hat S \, G^\top \hat x \, g^\top \hat S 
 - G \, \hat S \, G^\top \hat X G \, \hat S 
 + G \, \hat S \, G^\top \hat x \, \hat u^\top \\
 & & \quad
 + \hat B \, \hat x \, g^\top \hat S
 - \hat B \, \hat X G \, \hat S
 + \hat B \, \hat x \, \hat u^\top, \\
c \, \hat y^\top  & \stackrel{\eqref{opt_c_srr},\eqref{esrlt-opty}}{=} & 
 -G \, \hat u \, g^\top \hat S 
 + G \, \hat u \, \hat x^\top G \, \hat S
 -G \, \hat u \, \hat u^\top
 - G \, \hat S \, g \, g^\top \hat S 
  \\
 & & \quad + G \, \hat S \, g \, \hat x^\top G \, \hat S - G \, \hat S \, g \, \hat u^\top
+ \hat b \, g^\top \hat S 
- \hat b \, \hat x^\top G \, \hat S
+ \hat b \, \hat u^\top,
\\
G \, \hat Y & \stackrel{\eqref{esrlt-optY}}{=} & G \, \hat S \, G^\top \hat X G \, \hat S - G \, \hat S \, G^\top \hat x \, g^\top \hat S - G \, \hat S \, g \, \hat x^\top G \, \hat S + G \, \hat S \, g \, g^\top \hat S   \\ 
 & & \quad + G \, \hat S \, g \, \hat u^\top - G \, \hat S \, G^\top \hat x \, \hat u^\top + G \, \hat u \, g^\top \hat S -  G \, \hat u \, \hat x^\top G \, \hat S  + G \, \hat u \, \hat u^\top, \\
H \, \hat M_{yz}^\top & \stackrel{\eqref{esrlt-optW}}{=} & H (H^\top H)^{-1} H^\top \left( - \hat B \, \hat x \, \hat g^\top \hat S - \hat b \, \hat g^\top \hat S + \hat B \, \hat X G \, \hat S + \hat b \, \hat x^\top G \, \hat S \right) \\
 & & - H (H^\top H)^{-1} H^\top \left( \hat B \, \hat x \, \hat u^\top + \hat b \, \hat u^\top \right) \\
 & = & H (H^\top H)^{-1} H^\top \left( - \left( \hat B \, \hat x + \hat b \right) \left( \hat g^\top \hat S + \hat u^\top \right) + \left( \hat B \, \hat X + \hat b \, \hat x^\top \right) G \, \hat S \right) \\
 & \stackrel{\eqref{rel_zr},\eqref{rel_Wr}}{=} &  - \hat B \, \hat x \, \hat g^\top \hat S - \hat b \, \hat g^\top \hat S  - \hat B \, \hat x \, \hat u^\top - \hat b \, \hat u^\top + \hat B \, \hat X G \, \hat S + \hat b \, \hat x^\top G \, \hat S. 
\end{eqnarray*} 
It follows that 
$Q \, \hat M_{xy} + c \, \hat y^\top + G \, \hat Y + H \, \hat M_{yz}^\top = \mathbf{0}$.

Considering the tenth constraint, we obtain
\begin{eqnarray*}
Q \, \hat M_{xz} & \stackrel{\eqref{opt_Q_srr},\eqref{esrlt-optV}}{=} & - G \, \hat S \, G^\top \hat X \hat B^\top H (H^\top H)^{-1} - G \, \hat S \, G^\top \hat x \, \hat b^\top H (H^\top H)^{-1} \\ 
 & & \quad - \hat B \, \hat X \hat B^\top H (H^\top H)^{-1} - \hat B \, \hat x \, \hat b^\top H (H^\top H)^{-1}, \\
c \, \hat z^\top & \stackrel{\eqref{opt_c_srr},\eqref{minnorm_z}}{=} & G \, \hat u \, \hat x^\top \hat B^\top H (H^\top H)^{-1} + G \, \hat u \, \hat b^\top H (H^\top H)^{-1} \\ 
 & & \quad + G \, \hat S \, g \, \hat x^\top \hat B^\top H (H^\top H)^{-1} + G \, \hat S \, g \, \hat b^\top H (H^\top H)^{-1} \\
 & & \quad  - \hat b \, \hat x^\top \hat B^\top H (H^\top H)^{-1} - \hat b \, \hat b^\top H (H^\top H)^{-1}, \\
G \, \hat M_{yz} & \stackrel{\eqref{esrlt-optW}}{=} & - G \, \hat S \, g \, \hat x^\top \hat B^\top H (H^\top H)^{-1} - G \, \hat S \, g \, \hat b^\top H (H^\top H)^{-1} \\
 & & \quad 
+ G \, \hat S \, G^\top \hat X \hat B^\top H (H^\top H)^{-1} + G \, \hat S \, G^\top \hat x \, \hat b^\top H (H^\top H)^{-1} \\
 & & \quad 
- G \, \hat u \, \hat x^\top \hat B^\top H (H^\top H)^{-1} - G \, \hat u \, \hat b^\top H (H^\top H)^{-1}, \\ 
H \, \hat Z & \stackrel{\eqref{esrlt-optZ}}{=} & H (H^\top H)^{-1} H^\top \left( \hat B \, \hat X \hat B^\top + \hat B \, \hat x \, \hat b^\top + \hat b \, \hat x^\top \hat B^\top + \hat b \, \hat b^\top \right) H (H^\top H)^{-1} \\
 & = & H (H^\top H)^{-1} H^\top \left( \left( \hat B \, \hat X + \hat b \, \hat x^\top \right) \hat B^\top + \left(   \hat B \, \hat x + \hat b \right) \hat b^\top \right) H (H^\top H)^{-1} \\
 & \stackrel{\eqref{rel_zr},\eqref{rel_Wr}}{=} & \hat B \, \hat X \hat B^\top H (H^\top H)^{-1} + \hat b \, \hat x^\top \hat B^\top H (H^\top H)^{-1} \\
 & & \quad + \hat B \, \hat x \, \hat b^\top H (H^\top H)^{-1} + \hat b \, \hat b^\top H (H^\top H)^{-1}. 
\end{eqnarray*}
We conclude that 
$Q \, \hat M_{xz} + c \, \hat z^\top + G \, \hat M_{yz} + H \, \hat Z = \mathbf{0}$.

We next focus on the eleventh and twelfth constraints:
\begin{eqnarray*}
\hat y \, g^\top - \hat M_{xy}^\top \, G & \stackrel{\eqref{esrlt-opty},\eqref{esrlt-optU}}{=} & \hat S \, g \, g^\top - \hat S \, G^\top \hat x \, g^\top + \hat u \, g^\top - \hat S \, g \, \hat x^\top G + \hat S \, G^\top \hat X G - \hat u \, \hat x^\top G  \\
 & = & \hat u \left( g - G^\top \hat x \right)^\top + \hat S \left( G^\top \hat X G - G^\top \hat x \, g^\top - g \, \hat x^\top G + g \, g^\top \right). 
\end{eqnarray*}
Using \eqref{opt_u_srr}, the first constraint, \eqref{opt_S_srr}, and the fourth constraint, we conclude that $\hat y \, g^\top - \hat M_{xy}^\top \, G \geq \mathbf{0}$. By \eqref{opt_cs1_srr} and \eqref{opt_cs2_srr}, we obtain $\diag \left( \hat y \, g^\top - \hat M_{xy}^\top \, G \right) = \mathbf{0}$. 

Next, it follows from \eqref{esrlt-opty}, \eqref{esrlt-optY}, \eqref{opt_u_srr}, the first constraint, \eqref{opt_S_srr}, and the fourth constraint that $\hat y \geq \mathbf{0}$ and $\hat Y \geq \mathbf{0}$. Considering the fifteenth constraint, we arrive at 
\begin{eqnarray*}
    \langle Q, \hat X \rangle & \stackrel{\eqref{opt_Q_srr}}{=} & \langle \hat S, G^\top \hat X G \rangle + \langle \hat B, \hat X \rangle, \\
    c^\top \hat x & \stackrel{\eqref{opt_c_srr}}{=} & - \hat u^\top G^\top \hat x - g^\top \hat S \, G^\top \hat x + \hat b^\top \hat x, \\
    g^\top \hat y & \stackrel{\eqref{esrlt-opty}}{=} & g^\top \hat S \, g - g^\top \hat S \, G^\top \hat x + g^\top \hat u, \\
    h^\top \hat z & \stackrel{\eqref{minnorm_z}}{=} & -h^\top (H^\top H)^{-1} H^\top \hat B \, \hat x - h^\top (H^\top H)^{-1} H^\top \hat b.
\end{eqnarray*}
Substituting $H^\top \hat x = h$ in the fourth line, it follows that
\begin{eqnarray*}
\langle Q, \hat X \rangle + c^\top \hat x + g^\top \hat y + h^\top \hat z & = & \hat u^\top \left( g - G^\top \hat x \right) \\
 & & \quad + 
\left\langle \hat S, G^\top \hat X G - G^\top \hat x \, g^\top - g \, \hat x^\top G + g \, g^\top \right\rangle \\
 & & \quad + \langle \hat B, \hat X \rangle + \hat b^\top \hat x \\
& & \quad - \hat x^\top H (H^\top H)^{-1} H^\top \left( \hat B \, \hat x + \hat b \right)\\
 & \stackrel{\eqref{opt_u_srr},\eqref{opt_S_srr}}{=} & \langle \hat B, \hat X \rangle + \hat b^\top \hat x \\
  & & \quad - \hat x^\top H (H^\top H)^{-1} H^\top \left( \hat B \, \hat x + \hat b \right)\\
 & \stackrel{\eqref{rel_zr}}{=} & \langle \hat B, \hat X \rangle + \hat b^\top \hat x - \hat x^\top \hat B \, \hat x - \hat b^\top \hat x \\
 & = & \left\langle \hat B, \hat X - \hat x \, \hat x^\top \right\rangle \\
 & \stackrel{\eqref{rel_csr}}{=} & 0.
\end{eqnarray*}

Finally, since $\hat X - \hat x \hat x^\top \succeq 0$, we have 
\[
\begin{bmatrix}
 X & M_{xy} & M_{xz} \\
M_{xy}^\top & Y & M_{yz} \\
 M_{xz}^\top & M_{yz}^\top & Z 
 \end{bmatrix} - \begin{bmatrix} x \\ y \\ z \end{bmatrix} \begin{bmatrix} x \\ y \\ z \end{bmatrix}^\top = A \left( \hat X - \hat x \, \hat x^\top \right) A^\top \succeq \mathbf{0},
 \]
 where 
 \[
 A =  \begin{bmatrix} I \\ - S \, G^\top \\ - (H^\top H)^{-1} H^\top B \end{bmatrix} \in \R^{(n + m + p) \times n}.
 \]
We conclude that the semidefinite constraint is satisfied. 

Therefore, $(\hat x, \hat y, \hat z, \hat X, \hat Y, \hat Z, \hat M_{xy}, \hat M_{xz}, \hat M_{yz})$ is \eqref{esrlt}-feasible, which implies that $\nu^*_{\mathsf{SR+}} \leq \nu^*_{\mathsf{SR}}$. We conclude that $\nu^*_{\mathsf{SR+}} =  \nu^*_{\mathsf{SR}}$.

\end{proof}

\begin{corollary} \label{polytope_imp_srlt}
If $\mathsf{F}$ is nonempty and bounded, then $\nu^*_{\mathsf{SR+}} =  \nu^*_{\mathsf{SR}}$.    
\end{corollary}
\begin{proof}
For any $\hat x \in \mathsf{F}$, $(\hat x, \hat x \, \hat x^T)$ is \eqref{srlt}-feasible. 
Since the feasible region of \eqref{srlt} is a subset of the feasible region of \eqref{rlt}, we conclude that it is bounded by \cite[Lemma 7]{QiuY23a}, which implies that it is nonempty and compact. Therefore, $\nu^*_{\mathsf{SR}}$ is finite and the set of optimal solutions is nonempty. The assertion follows from Theorem~\ref{srltcsesrlt}.

\end{proof}

Theorem~\ref{srltcsesrlt} establishes the equivalence of the SDP-RLT relaxations \eqref{srlt} and \eqref{esrlt} under the assumption that the optimal value is attained in \eqref{srlt}. We close this section with a brief discussion of instances of \eqref{QP} that admit an unbounded SDP-RLT relaxation. 

Let us consider the same families of instances of \eqref{QP} in Examples~\ref{ubex0}--\ref{ubex3}. It is well-known and easy to verify that the Shor relaxation is exact whenever the objective function is convex. Therefore, in Examples~\ref{ubex0} and~\ref{ubexparam1}, we obtain $\nu^*_{\mathsf{SR+}} (\alpha) =  \nu^*_{\mathsf{SR}} (\alpha) = \nu^* (\alpha)$ for each $\alpha \in \R$ by Theorem~\ref{srltcsesrlt}. In Examples~\ref{ubex2} and~\ref{ubex3}, we have $\nu^*_{\mathsf{SR}}(\alpha) = \nu^*(\alpha) = -\infty$ by \eqref{srltrlt-lb}.
On the other hand, recall that \eqref{QP} and \eqref{eqp} are no longer equivalent in these examples. Since the RLT relaxation \eqref{erlt} is already an exact relaxation of \eqref{eqp} (see the discussion at the end of Section~\ref{Sec3}), it follows from \eqref{srltrlt-lb} that $-\infty = \nu^*(\alpha) < \nu^*_{\mathsf{+}}(\alpha) = \nu^*_{\mathsf{R+}}(\alpha) = \nu^*_{\mathsf{SR+}}(\alpha)$ for each $\alpha \in \R$. Therefore, similar to the RLT relaxation \eqref{erlt}, the SDP-RLT relaxation \eqref{esrlt} is no longer a valid relaxation of \eqref{QP}. Unless a quadratic program is known to have a finite optimal value, we conclude that relaxations of optimality conditions should be considered with some caveat.

\section{Concluding Remarks} \label{Sec5}

In this paper, we considered whether the RLT and SDP-RLT relaxations of general quadratic programs can be strengthened by incorporating optimality conditions. We established a negative answer under the assumption that the RLT and SDP-RLT bounds are finite. 

For instances of quadratic programs with unbounded RLT or SDP-RLT relaxations, we presented several families of examples that illustrate different scenarios. In particular, if the quadratic program does not have a finite optimal value, our examples reveal that relaxations of the optimality conditions may even yield misleading information about the original problem. 

We considered the effect of relaxations of first-order optimality conditions. By exploiting optimality conditions for box-constrained quadratic programs, Burer and Chen~\cite{BurerC11} propose a novel semidefinite programming relaxation that incorporates additional information from the second-order optimality conditions. We intend to investigate whether such a construction could be extended to general quadratic programs and whether the resulting relaxation remains equivalent to the original semidefinite relaxation.













\section*{Conflict of interest}

The authors declare that they have no conflict of interest.

\section*{Availability of Data and Materials}

The manuscript has no associated data.

\bibliographystyle{spmpsci}      
\bibliography{sn-bibliography} 

\begin{thebibliography}{10}
\providecommand{\url}[1]{{#1}}
\providecommand{\urlprefix}{URL }
\expandafter\ifx\csname urlstyle\endcsname\relax
  \providecommand{\doi}[1]{DOI~\discretionary{}{}{}#1}\else
  \providecommand{\doi}{DOI~\discretionary{}{}{}\begingroup \urlstyle{rm}\Url}\fi

\bibitem{AudetHJS99}
Audet, C., Hansen, P., Jaumard, B., Savard, G.: A symmetrical linear maxmin approach to disjoint bilinear programming.
\newblock Math. Program. \textbf{85}(3), 573--592 (1999).
\newblock \doi{10.1007/S101070050072}.
\newblock \urlprefix\url{https://doi.org/10.1007/s101070050072}

\bibitem{BaoST11}
Bao, X., Sahinidis, N.V., Tawarmalani, M.: Semidefinite relaxations for quadratically constrained quadratic programming: {A} review and comparisons.
\newblock Math. Program. \textbf{129}(1), 129--157 (2011).
\newblock \doi{10.1007/s10107-011-0462-2}.
\newblock \urlprefix\url{https://doi.org/10.1007/s10107-011-0462-2}

\bibitem{deB90}
de~Boor, C.: An empty exercise.
\newblock SIGNUM Newsl. \textbf{25}(4), 2–6 (1990).
\newblock \doi{10.1145/122272.122273}.
\newblock \urlprefix\url{https://doi.org/10.1145/122272.122273}

\bibitem{BW1981b}
Borwein, J., Wolkowicz, H.: Regularizing the abstract convex program.
\newblock Journal of Mathematical Analysis and Applications \textbf{83}(2), 495--530 (1981).
\newblock \doi{10.1016/0022-247X(81)90138-4}.
\newblock \urlprefix\url{https://www.sciencedirect.com/science/article/pii/0022247X81901384}

\bibitem{BW1981a}
Borwein, J.M., Wolkowicz, H.: Facial reduction for a cone-convex programming problem.
\newblock Journal of the Australian Mathematical Society. Series A. Pure Mathematics and Statistics \textbf{30}(3), 369–380 (1981).
\newblock \doi{10.1017/S1446788700017250}.
\newblock \urlprefix\url{https://doi.org/10.1017/S1446788700017250}

\bibitem{BurerC11}
Burer, S., Chen, J.: Relaxing the optimality conditions of box {QP}.
\newblock Comput. Optim. Appl. \textbf{48}(3), 653--673 (2011).
\newblock \doi{10.1007/S10589-009-9273-2}.
\newblock \urlprefix\url{https://doi.org/10.1007/s10589-009-9273-2}

\bibitem{BurerV08}
Burer, S., Vandenbussche, D.: A finite branch-and-bound algorithm for nonconvex quadratic programming via semidefinite relaxations.
\newblock Math. Program. \textbf{113}(2), 259--282 (2008).
\newblock \doi{10.1007/S10107-006-0080-6}.
\newblock \urlprefix\url{https://doi.org/10.1007/s10107-006-0080-6}

\bibitem{BurerV09}
Burer, S., Vandenbussche, D.: Globally solving box-constrained nonconvex quadratic programs with semidefinite-based finite branch-and-bound.
\newblock Comput. Optim. Appl. \textbf{43}(2), 181--195 (2009).
\newblock \doi{10.1007/S10589-007-9137-6}.
\newblock \urlprefix\url{https://doi.org/10.1007/s10589-007-9137-6}

\bibitem{ChenB12}
Chen, J., Burer, S.: Globally solving nonconvex quadratic programming problems via completely positive programming.
\newblock Math. Program. Comput. \textbf{4}(1), 33--52 (2012).
\newblock \doi{10.1007/S12532-011-0033-9}.
\newblock \urlprefix\url{https://doi.org/10.1007/s12532-011-0033-9}

\bibitem{DrusvyatskiyW17}
Drusvyatskiy, D., Wolkowicz, H.: The many faces of degeneracy in conic optimization.
\newblock Found. Trends Optim. \textbf{3}(2), 77--170 (2017).
\newblock \doi{10.1561/2400000011}.
\newblock \urlprefix\url{https://doi.org/10.1561/2400000011}

\bibitem{FW1956}
Frank, M., Wolfe, P.: An algorithm for quadratic programming.
\newblock Naval Research Logistics Quarterly \textbf{3}(1-2), 95--110 (1956).
\newblock \doi{10.1002/nav.3800030109}.
\newblock \urlprefix\url{https://doi.org/10.1002/nav.3800030109}

\bibitem{FuriniTBFGGLLMM19}
Furini, F., Traversi, E., Belotti, P., Frangioni, A., Gleixner, A.M., Gould, N., Liberti, L., Lodi, A., Misener, R., Mittelmann, H.D., Sahinidis, N.V., Vigerske, S., Wiegele, A.: {QPLIB:} {A} library of quadratic programming instances.
\newblock Math. Program. Comput. \textbf{11}(2), 237--265 (2019).
\newblock \doi{10.1007/s12532-018-0147-4}.
\newblock \urlprefix\url{https://doi.org/10.1007/s12532-018-0147-4}

\bibitem{GT73}
Giannessi, F., Tomasin, E.: Nonconvex quadratic programs, linear complementarity problems, and integer linear programs.
\newblock In: R.~Conti, A.~Ruberti (eds.) 5th Conference on Optimization Techniques Part I, pp. 437--449. Springer Berlin Heidelberg, Berlin, Heidelberg (1973).
\newblock \doi{0.1007/3-540-06583-0_43}.
\newblock \urlprefix\url{https://doi.org/10.1007/3-540-06583-0_43}

\bibitem{GondzioY21}
Gondzio, J., Yildirim, E.A.: Global solutions of nonconvex standard quadratic programs via mixed integer linear programming reformulations.
\newblock J. Glob. Optim. \textbf{81}(2), 293--321 (2021).
\newblock \doi{10.1007/S10898-021-01017-Y}.
\newblock \urlprefix\url{https://doi.org/10.1007/s10898-021-01017-y}

\bibitem{HPBR93}
Hansen, P., Jaumard, B., Ruiz, M., Xiong, J.: Global minimization of indefinite quadratic functions subject to box constraints.
\newblock Naval Research Logistics (NRL) \textbf{40}(3), 373--392 (1993).
\newblock \doi{10.1002/1520-6750(199304)40:3<373::AID-NAV3220400307>3.0.CO;2-A}.
\newblock \urlprefix\url{https://doi.org/10.1002/1520-6750(199304)40:3<373::AID-NAV3220400307>3.0.CO;2-A}

\bibitem{HuMP12}
Hu, J., Mitchell, J.E., Pang, J.: An {LPCC} approach to nonconvex quadratic programs.
\newblock Math. Program. \textbf{133}(1-2), 243--277 (2012).
\newblock \doi{10.1007/S10107-010-0426-Y}.
\newblock \urlprefix\url{https://doi.org/10.1007/s10107-010-0426-y}

\bibitem{LS2013}
Locatelli, M., Schoen, F.: Global Optimization - Theory, Algorithms, and Applications, \emph{{MOS-SIAM} Series on Optimization}, vol.~15.
\newblock {SIAM} (2013).
\newblock \doi{10.1137/1.9781611972672}.
\newblock \urlprefix\url{https://doi.org/10.1137/1.9781611972672}

\bibitem{LuoT92}
Luo, Z., Tseng, P.: Error bound and convergence analysis of matrix splitting algorithms for the affine variational inequality problem.
\newblock {SIAM} J. Optim. \textbf{2}(1), 43--54 (1992).
\newblock \doi{10.1137/0802004}.
\newblock \urlprefix\url{https://doi.org/10.1137/0802004}

\bibitem{Pardalos1991QuadraticPW}
Pardalos, P.M., Vavasis, S.A.: Quadratic programming with one negative eigenvalue is {NP}-hard.
\newblock Journal of Global Optimization \textbf{1}, 15--22 (1991).
\newblock \doi{10.1007/BF00120662}.
\newblock \urlprefix\url{https://doi.org/10.1007/BF00120662}

\bibitem{Pat2013}
Pataki, G.: Strong duality in conic linear programming: Facial reduction and extended duals.
\newblock In: D.H. Bailey, H.H. Bauschke, P.~Borwein, F.~Garvan, M.~Th{\'e}ra, J.D. Vanderwerff, H.~Wolkowicz (eds.) Computational and Analytical Mathematics, pp. 613--634. Springer New York, New York, NY (2013).
\newblock \doi{10.1007/978-1-4614-7621-4_28}.
\newblock \urlprefix\url{https://doi.org/10.1007/978-1-4614-7621-4_28}

\bibitem{PermenterP18}
Permenter, F., Parrilo, P.A.: Partial facial reduction: simplified, equivalent sdps via approximations of the {PSD} cone.
\newblock Math. Program. \textbf{171}(1-2), 1--54 (2018).
\newblock \doi{10.1007/S10107-017-1169-9}.
\newblock \urlprefix\url{https://doi.org/10.1007/s10107-017-1169-9}

\bibitem{QiuY23a}
Qiu, Y., \Yildirim, E.A.: Polyhedral properties of {RLT} relaxations of nonconvex quadratic programs and their implications on exact relaxations.
\newblock Mathematical Programming \textbf{209}, 397--433 (2025).
\newblock \doi{10.1007/s10107-024-02070-7}.
\newblock \urlprefix\url{https://doi.org/10.1007/s10107-024-02070-7}

\bibitem{rockafellar1970convex}
Rockafellar, R.T.: Convex Analysis.
\newblock Princeton Landmarks in Mathematics and Physics. Princeton University Press (1970)

\bibitem{Sahni74}
Sahni, S.: Computationally related problems.
\newblock {SIAM} J. Comput. \textbf{3}(4), 262--279 (1974).
\newblock \doi{10.1137/0203021}.
\newblock \urlprefix\url{https://doi.org/10.1137/0203021}

\bibitem{Sherali1999}
Sherali, H.D., Adams, W.P.: A Reformulation-Linearization Technique for Solving Discrete and Continuous Nonconvex Problems.
\newblock Springer US, Boston, MA (1999).
\newblock \doi{10.1007/978-1-4757-4388-3}.
\newblock \urlprefix\url{https://doi.org/10.1007/978-1-4757-4388-3}

\bibitem{sherali1995reformulation}
Sherali, H.D., Tun\c{c}bilek, C.H.: A reformulation-convexification approach for solving nonconvex quadratic programming problems.
\newblock Journal of Global Optimization \textbf{7}(1), 1--31 (1995).
\newblock \doi{10.1007/BF01100203}.
\newblock \urlprefix\url{https://doi.org/10.1007/BF01100203}

\bibitem{shor1987approach}
Shor, N.Z.: An approach to obtaining global extremums in polynomial mathematical programming problems.
\newblock Cybernetics \textbf{23}(5), 695--700 (1987)

\bibitem{VandenbusscheN05a}
Vandenbussche, D., Nemhauser, G.L.: A branch-and-cut algorithm for nonconvex quadratic programs with box constraints.
\newblock Math. Program. \textbf{102}(3), 559--575 (2005).
\newblock \doi{10.1007/S10107-004-0550-7}.
\newblock \urlprefix\url{https://doi.org/10.1007/s10107-004-0550-7}

\bibitem{VandenbusscheN05}
Vandenbussche, D., Nemhauser, G.L.: A polyhedral study of nonconvex quadratic programs with box constraints.
\newblock Math. Program. \textbf{102}(3), 531--557 (2005).
\newblock \doi{10.1007/S10107-004-0549-0}.
\newblock \urlprefix\url{https://doi.org/10.1007/s10107-004-0549-0}

\bibitem{WakiM13}
Waki, H., Muramatsu, M.: Facial reduction algorithms for conic optimization problems.
\newblock J. Optim. Theory Appl. \textbf{158}(1), 188--215 (2013).
\newblock \doi{10.1007/S10957-012-0219-Y}.
\newblock \urlprefix\url{https://doi.org/10.1007/s10957-012-0219-y}

\bibitem{XiaVZ20}
Xia, W., Vera, J., Zuluaga, L.F.: Globally solving nonconvex quadratic programs via linear integer programming techniques.
\newblock {INFORMS} J. Comput. \textbf{32}(1), 40--56 (2020).
\newblock \doi{10.1287/IJOC.2018.0883}.
\newblock \urlprefix\url{https://doi.org/10.1287/ijoc.2018.0883}

\end{thebibliography}

\end{document}